\documentclass[11pt]{article}

\usepackage{geometry}
\usepackage{latexsym}
\usepackage{mathrsfs}
\usepackage{amsfonts}
\usepackage{amssymb}
\usepackage{subfigure}    
\usepackage{graphicx}
\usepackage{epstopdf}
\usepackage{float}
\usepackage{multirow}
\usepackage{booktabs}
\usepackage{threeparttable}
\usepackage{bm}
\usepackage{amsmath}
\usepackage{amsthm}
\usepackage{color}
\usepackage[subnum]{cases}
\usepackage{rotating}
\usepackage{enumitem}
\usepackage{accents}
\usepackage{algorithm}
\usepackage{algorithmic}
\usepackage[title]{appendix}
\usepackage{mathtools}
\usepackage{caption}
\usepackage{url}
\usepackage{comment}
\usepackage{natbib}
\makeatletter
\@addtoreset{equation}{section}
\makeatother

\makeatletter
\def \wideubar{\underaccent{{\cc@style\underline{\mskip15mu}}}}
\def \widebar{\accentset{{\cc@style\underline{\mskip10mu}}}}
\makeatother

\definecolor{lightgreen}{rgb}{0.1,0.5,0.1}
\definecolor{revisionblue}{rgb}{0,0,0}

\usepackage[colorlinks=true,
breaklinks=true,
bookmarks=true,
urlcolor=blue,
citecolor=lightgreen,
linkcolor=lightgreen,
bookmarksopen=false,
draft=false]{hyperref}

\graphicspath{{images/}}

\def\iHPGPPA{{iHP-GPPA }}
\def\Frho{\mathcal{F}_{\rho}}

\begin{document}

\newtheorem{property}{Property}[section]
\newtheorem{proposition}{Proposition}[section]
\newtheorem{append}{Appendix}[section]
\newtheorem{definition}{Definition}[section]
\newtheorem{lemma}{Lemma}[section]
\newtheorem{corollary}{Corollary}[section]
\newtheorem{theorem}{Theorem}[section]
\newtheorem{remark}{Remark}[section]
\newtheorem{problem}{Problem}[section]
\newtheorem{example}{Example}[section]
\newtheorem{assumption}{Assumption}
\renewcommand*{\theassumption}{\Alph{assumption}}

\renewcommand{\vartheta}{\widehat{\theta}} 
\def\myfbox#1{\fbox{\begin{minipage}{15cm}{#1}\end{minipage}}\\ }

\title{An Inexact Halpern-accelerated {Preconditioned}
\\
Generalized Proximal Point Algorithm\\
for the Maximal Monotone Inclusion Problem}

\author{
Lei Yang\thanks{School of Computer Science and Engineering, Sun Yat-sen University ({\tt yanglei39@mail.sysu.edu.cn}).},~~
Haihang Lan\thanks{School of Computer Science and Engineering, Sun Yat-sen University ({\tt lanhh3@mail2.sysu.edu.cn}).},~~
Di Hou\thanks{(Corresponding author) Department of Mathematics, National University of Singapore ({\tt dihou@u.nus.edu}).},~~
Ling Liang\thanks{Department of Mathematics, University of Tennessee, Knoxville ({\tt liang.ling@u.nus.edu}).},~~
Kim-Chuan Toh\thanks{Department of Mathematics, and Institute of Operations Research and Analytics, National University of Singapore ({\tt mattohkc@nus.edu.sg}).}
}
\date{September 20, 2026}

\maketitle

\begin{abstract}
This paper studies an inexact Halpern-accelerated generalized proximal point algorithm with an admissible positive semidefinite preconditioner $\mathcal{M}$ for solving maximal monotone inclusion problems. The proposed framework combines Halpern anchoring, resolvent inexactness, and relaxation over the full weight range $\rho\in(0,2]$, thereby covering both under-relaxed and over-relaxed proximal iterations within a unified scheme. We first establish convergence of the inexact resolvent sequence under conditions that allow general anchoring parameters and certain nonsummable tolerances. For anchoring parameters $\beta_{k}=1/(k+r)$ with $r\geq2$, we then derive explicit bounds on the squared fixed-point residual in the $\mathcal{M}$-seminorm. In particular, if the tolerances satisfy $\varepsilon_{k}=\mathcal{O}((k+1)^{-\alpha})$ with $\alpha>3/2$ for $0<\rho<2$ and $\alpha>2$ for $\rho=2$, these bounds yield an $\mathcal{O}(1/k^{2})$ convergence rate. The stronger decay condition on the inexactness tolerances at $\rho=2$ highlights a qualitative distinction between the endpoint and the interior regime in the inexact setting. Finally, we develop inexact accelerated versions of the preconditioned alternating direction method of multipliers (pADMM) and the preconditioned primal--dual hybrid gradient (PDHG) method based on this framework, derive inexactness criteria based on subproblem residuals, and establish a nonergodic $\mathcal{O}(1/k)$ KKT residual rate for their inexact iterates.

\vspace{5mm}

\end{abstract}

\section{Introduction}

In this paper, we consider the following maximal monotone inclusion problem:
\begin{equation}\label{pro-inclu}
\text{Find}\quad \boldsymbol{x}^*\in \mathbb{R}^n
\quad \text{such that} \quad
0\in \mathcal{A}\boldsymbol{x}^*,
\end{equation}
where $\mathbb{R}^n$ is the $n$-dimensional Euclidean space and $\mathcal{A}:\mathbb{R}^n\rightrightarrows\mathbb{R}^n$ is a maximally monotone operator. This formulation underlies a broad class of convex optimization models and variational problems \cite{bc2017Convex, bauschke2025bredies}. The proximal point algorithm (PPA) \cite{martinet1970breve} is one of the most fundamental methods for solving \eqref{pro-inclu}, and many classical algorithms, such as the augmented Lagrangian method (ALM) \cite{rockafellar1978monotone}, the alternating direction method of multipliers (ADMM) \cite{eb1992douglas}, and the primal--dual hybrid gradient (PDHG) method \cite{zhu2008efficient,esser2010general,chambolle2011first}, can be interpreted as special realizations or approximations of PPA.

Given a proximal parameter $\lambda >0$, the exact PPA generates a sequence via
\begin{equation*}
\boldsymbol{x}^{k+1} = \mathcal{J}_{\lambda\mathcal{A}}\boldsymbol{x}^k, \quad k\geq0,
\end{equation*}
where $\mathcal{J}_{\lambda\mathcal{A}} := (\mathcal{I}+\lambda\mathcal{A})^{-1}$ denotes the resolvent of $\mathcal{A}$ and $\mathcal{I}$ is the identity map. Despite its generality, the classical PPA is hindered by two key limitations: (1) the inherent challenge of exact resolvent computation; and (2) a relatively slow worst-case convergence rate of $\mathcal{O}(1/k)$ for the \textit{squared} fixed-point residual, i.e., $\|\boldsymbol{x}^{k}-\mathcal{J}_{\lambda \mathcal{A}}\boldsymbol{x}^k\|^2$ \cite{baillion1996rate, gu2020tight}.

To improve the practical performance of PPA, various generalized (or relaxed) variants of PPA have been developed. A seminal contribution in this direction is due to Gol'shtein and Tret'yakov \cite{Golstein1979Modified}, who proposed the following inexact generalized proximal point algorithm (iGPPA), which combines relaxation with {inexact} resolvent evaluations:
\begin{equation}\label{algo-iGPPA}
\begin{aligned}
&\text{Find}~~{\boldsymbol{y}}^k
~~\text{such that}~~ \big\|{\boldsymbol{y}}^k-\mathcal{J}_{\lambda\mathcal{A}}\boldsymbol{x}^k\big\| \leq \varepsilon_k, \\
&\boldsymbol{x}^{k+1} = \rho_k{\boldsymbol{y}}^k+(1-\rho_k)\boldsymbol{x}^k, \quad k \geq 0,
\end{aligned}
\end{equation}
where $\varepsilon_k\ge 0$ controls the inexactness in evaluating the resolvent $\mathcal{J}_{\lambda\mathcal{A}}(\boldsymbol{x}^k)$ and $\rho_k\in(0,2)$ is a relaxation parameter. Subsequently, Eckstein and Bertsekas \cite{eb1992douglas} allowed for a variable proximal parameter $\lambda_k$ in \eqref{algo-iGPPA} and established the convergence of iGPPA. For iGPPA, allowing the over-relaxation $\rho_k>1$ is often empirically beneficial and has been widely adopted in practice. Nevertheless, an aggressive over-relaxation, particularly when $\rho_k$ closely approaches 2, may lead to numerical instability, which limits its practical effectiveness. For fixed $\lambda>0$ and $\rho_k=\rho\in(0,2)$, with $\varepsilon_k=\mathcal{O}((k+1)^{-\alpha})$ for $\alpha>1$, Corman and Yuan \cite{cy2014generalized} established an $\mathcal{O}(1/k)$ convergence rate for the squared fixed-point residual for iGPPA. Besides relaxation, acceleration techniques inspired by Nesterov and G\"uler \cite{nesterov1983method, guler1992new} have also been investigated for proximal methods \cite{alvarez2001inertial,attouch2019convergence, attouch2020convergence}.

A breakthrough was recently achieved by Kim \cite{kim2021accelerated}, who proposed an accelerated PPA inspired by Nesterov’s acceleration techniques. Using the performance estimation problem (PEP) framework, Kim established an accelerated convergence rate of $\mathcal{O}(1/k^2)$ for the squared fixed-point residual. Subsequently, Contreras and Cominetti \cite{cc2023optimal} provided a deeper structural interpretation by showing that Kim’s accelerated PPA can be viewed as a Halpern iteration \cite{halpern1967fixed} applied to a resolvent-based fixed-point mapping.

The Halpern iteration was originally designed for finding fixed points of nonexpansive mappings in Hilbert spaces, and offers strong convergence guarantees and an accelerated rate. In particular, Lieder \cite{l2021convergence-Halpern} derived an explicit and optimal $\mathcal{O}(1/k^2)$ convergence rate for the squared fixed-point residual via a direct analysis. Thus, its use as an acceleration mechanism for optimization algorithms has attracted renewed attention in recent years; see, for example, \cite{ly2024restarted,syzz2024acc-pADMM}.

For any fixed $\rho\in(0,2]$, define the relaxed resolvent mapping
\begin{equation*}
\mathcal{T}_{\rho}:=\rho\mathcal{J}_{\lambda\mathcal{A}}+(1-\rho)\mathcal{I}.
\end{equation*}
Since $\mathcal{T}_{\rho}$ is nonexpansive on the whole interval $\rho\in(0,2]$, it is natural to apply the Halpern iteration to this entire relaxed family rather than only to the proximal case $\rho=1$. This leads to the following Halpern-accelerated generalized PPA:
\begin{equation}\label{Hal-gene}
\boldsymbol{x}^{k+1} = \beta_k \boldsymbol{x}^0 + (1 - \beta_k)\big[\rho \mathcal{J}_{\lambda\mathcal{A}}\boldsymbol{x}^k+(1-\rho)\boldsymbol{x}^k\big], \quad k \geq 0,
\end{equation}
where $\boldsymbol{x}^0$ is a fixed anchor point. With the classical choice $\beta_k=\frac{1}{k+2}$, the Halpern mechanism improves the worst-case convergence rate of the squared fixed-point residual from $\mathcal{O}(1/k)$ to $\mathcal{O}(1/k^2)$. The endpoint $\rho=2$ corresponds to the reflected resolvent $2\mathcal{J}_{\lambda\mathcal{A}}-\mathcal{I}$, while the interior regime $\rho\in(0,2)$ continuously interpolates between under-relaxation and over-relaxation. From the exact Halpern viewpoint, the reflected endpoint is especially attractive: it allows one to apply tight worst-case residual estimates for nonexpansive mappings directly to the reflected resolvent formulation; see \cite{cc2023optimal,l2021convergence-Halpern}. This unified viewpoint is central to the present paper, because it allows us to analyze how relaxation interacts with inexactness and anchoring within a single framework.

Despite these appealing exact-case properties, the practical deployment of \eqref{Hal-gene} raises several nontrivial challenges.

\textbf{Inexact resolvent evaluation}. In many practical applications, computing the resolvent operator exactly is computationally expensive or even infeasible. This motivates the study of inexact variants of the Halpern-accelerated relaxed scheme \eqref{Hal-gene}. A basic inexact scheme can be given as follows:
\begin{equation*}
\begin{aligned}
&\text{Find}~~{\boldsymbol{y}}^k
~~\text{such that}~~ \big\|{\boldsymbol{y}}^k-\mathcal{J}_{\lambda\mathcal{A}}\boldsymbol{x}^k\big\| \leq \varepsilon_k, \\
&\boldsymbol{x}^{k+1} = \beta_k \boldsymbol{x}^0 + (1 - \beta_k)\big[\rho{\boldsymbol{y}}^k+(1-\rho)\boldsymbol{x}^k\big], \quad k \geq 0.
\end{aligned}
\end{equation*}
For $\rho=1$, the scheme above reduces to the inexact Halpern-accelerated PPA (iH-PPA), which has been studied in the literature. Early works \cite{kt2000approximating,mx2004convergence,x2002Iterative} established its convergence under summable errors and suitable parameter conditions; however, no explicit residual rate was derived. Recently, for the classical choice $\beta_{k}=1/(k+2)$, Diakonikolas \cite{d2020halpern} showed that resolvent tolerances $\varepsilon_{k}\leq\epsilon/[4(k+1)(k+2)]$ suffice for iH-PPA to attain a fixed-point residual of at most $\epsilon>0$ within $\mathcal{O}(1/\epsilon)$ iterations. Later, Liang et al. \cite{ltz2024inexact} studied inexact Halpern iterations for general cocoercive operators. In the proximal specialization corresponding to $\rho=1$, their results show that, under tolerances $\varepsilon_{k}=\mathcal{O}((k+1)^{-\alpha})$ with $\alpha>3/2$, iH-PPA achieves an $\mathcal{O}(1/k^{2})$ rate for the squared fixed-point residual. More recently, Zhang et al. \cite{zhang2025halpern} studied iH-PPA using a different analytical approach and established an $\mathcal{O}(1/k^{2})$ rate for the squared fixed-point residual under the more stringent decay condition $\varepsilon_{k}=\mathcal{O}((k+1)^{-\alpha})$ with $\alpha>3$.

Motivated by these developments, we investigate how the relaxation parameter $\rho\in(0,2]$ affects the error conditions sufficient to preserve the accelerated rate for the squared fixed-point residual. Our analysis distinguishes the interior regime $0<\rho<2$ from the endpoint $\rho=2$, for which we establish a sharp weighted summability condition.

\textbf{Anchoring parameter selection}. The choice of anchoring parameters $\beta_k$ plays a crucial role in the behavior of Halpern iterations. Although the classical choice $\beta_k=\frac{1}{k+2}$ is theoretically optimal in several exact settings, it may be suboptimal in practice. Therefore, allowing tunable anchoring parameters can potentially lead to improved numerical performance in Halpern-type methods across different values of $\rho$. Suh et al. \cite{suh2023continuous} derived accelerated convergence rates for exact Halpern-type proximal methods using anchoring parameters of the form $\beta_k=\frac{\gamma}{k+\gamma}$ with $\gamma\geq1$, as well as adaptive anchoring strategies, based on an analysis via ordinary differential equations. Along similar lines, He et al. \cite{hxlm2025convergence} proposed Halpern iterations with adaptive anchoring parameters and established improved convergence guarantees. In the inexact relaxed setting, {we establish accelerated residual bounds for $\beta_k=\frac{1}{k+r}$ with $r\geq2$ that explicitly account for the anchoring parameter, relaxation, and inexactness tolerances}.

\textbf{Preconditioned {PPA}}. The preconditioned PPA incorporates a preconditioner into the vanilla PPA to adjust the structure of the resolvent, thus constituting a generalized form of the latter, which has been shown in \cite{bcln2022degenerate, xue2026degenerate} to encompass a wide range of first-order operator splitting algorithms. These include the Douglas–Rachford splitting (DRS) \cite{eb1992douglas} and its preconditioned variants \cite{bredies2015preconditioned}, ADMM \cite{boyd2011distributed} and its preconditioned variants \cite{bredies2017proximal}, and the Chambolle–Pock algorithm \cite{chambolle2011first}. The flexible selection of preconditioners broadens the applicability of these algorithms. Recently, for positive semidefinite (degenerate) preconditioners, Zhang et al. \cite{zhang2025hppp} incorporated Halpern acceleration into the preconditioned PPA and analyzed its convergence and accelerated convergence rate. Sun et al. \cite{syzz2024acc-pADMM} leveraged the connection between the preconditioned PPA and the preconditioned alternating direction method of multipliers (pADMM) to establish an accelerated pADMM.

The developments above motivate the joint study of inexact computation, adjustable anchoring, and relaxation over $\rho\in(0,2]$ in a preconditioned setting. We integrate these aspects within a unified and practically implementable framework. Specifically, we study an inexact Halpern-accelerated preconditioned generalized proximal point algorithm (iHP-GPPA) that allows for inexact resolvent evaluation, a relaxation parameter $\rho\in(0,2]$, an adjustable anchoring parameter sequence, and an admissible positive semidefinite preconditioner $\mathcal{M}$.

The main contributions of this work can be summarized as follows:
\begin{itemize}
    \item We establish convergence of the inexact resolvent sequence $\{\boldsymbol{y}^{k}\}$ for all $\rho\in(0,2]$ under conditions allowing general anchoring parameters and certain nonsummable tolerances. Its limit is the unique solution nearest to the anchor $\boldsymbol{x}^{0}$ in the $\mathcal{M}$-seminorm. For $0<\rho<2$, we also establish the convergence of $\{\boldsymbol{x}^k\}$ to the same solution, even when $\mathcal{M}$ is singular.

    \item For anchoring parameters $\beta_{k}=1/(k+r)$ with $r\geq2$, we derive explicit bounds that quantify the effects of anchoring, relaxation, and inexactness on the squared fixed-point residual in the $\mathcal{M}$-seminorm. In particular, an $\mathcal{O}(1/k^{2})$ rate holds if $\sum_{k=0}^{\infty}(k+r-1)^{2}\varepsilon_{k}^{2}<\infty$ for $0<\rho<2$, or $\sum_{k=0}^{\infty}(k+r-1)\varepsilon_{k}<\infty$ for $\rho=2$. Moreover, at $\rho=2$, the weighted summability condition is also necessary for guaranteeing this rate uniformly over all admissible error directions.

    \item Under the corresponding weighted summability conditions and Lipschitz continuity of $(\mathcal{M}+\mathcal{A})^{-1}$, we establish an $\mathcal{O}(1/k^{2})$ rate for the squared Euclidean fixed-point residual at $\boldsymbol{y}^{k}$ for all $\rho\in(0,2]$, and at $\boldsymbol{x}^{k}$ for $0<\rho<2$. A counterexample shows that, with singular $\mathcal{M}$, convergence of $\{\boldsymbol{x}^{k}\}$ and its Euclidean residual rate can fail at $\rho=2$.

    \item We develop inexact Halpern-accelerated versions of pADMM and preconditioned PDHG with computable inexactness criteria based on subproblem residuals. Under these criteria and the corresponding weighted summability conditions, we establish a nonergodic $\mathcal{O}(1/k)$ KKT residual rate for their inexact iterates over the full relaxation range $\rho\in(0,2]$.

\end{itemize}

The remainder of this paper is organized as follows. Section \ref{sec-not-pre} provides notation and preliminaries. Section \ref{sec-iAccGPPA} presents the proposed \iHPGPPA framework. Section \ref{sec-convergence} presents its convergence and convergence rate analysis. Section \ref{sec-App} presents its applications to inexact accelerated pADMM and preconditioned PDHG. Section \ref{sec-conclusions} concludes the paper.

\section{Notation and preliminaries}\label{sec-not-pre}

We adopt the following notation throughout the paper for consistency and clarity.
\begin{itemize}[leftmargin=0pt, itemsep=0pt, topsep=5pt]

\item[] \textbf{Monotone operators and solution sets}: Let $\mathcal{A}: \mathbb{R}^n \rightrightarrows \mathbb{R}^n$ be a maximally monotone operator. The solution set of the inclusion problem \eqref{pro-inclu} is denoted by
\[
\mathcal{S} := \{\boldsymbol{x} \in \mathbb{R}^n \mid 0 \in \mathcal{A}\boldsymbol{x}\}.
\]

\item[] \textbf{Resolvent and auxiliary operators}: For $\lambda > 0$, the resolvent of $\mathcal{A}$ is defined as $\mathcal{J}_{\lambda\mathcal{A}} := (\mathcal{I} + \lambda\mathcal{A})^{-1}$, where $\mathcal{I}$ denotes the identity operator on $\mathbb{R}^n$. For $\rho \in [0, 2]$, we define the following auxiliary operator:
\[
 \mathcal{T}_{\rho} := \rho\mathcal{J}_{\lambda\mathcal{A}} + (1 - \rho)\mathcal{I}.
\]

\item[] \textbf{Metric projection}: For a nonempty closed convex subset $C \subset \mathbb{R}^n$, the metric projection $P_C: \mathbb{R}^n \to C$ is given by
\[
P_C(\boldsymbol{x}) = \arg\min\left\{\|\boldsymbol{x} - \boldsymbol{y}\| \mid \boldsymbol{y} \in C\right\}, \quad \forall \boldsymbol{x} \in \mathbb{R}^n.
\]

\item[] \textbf{Convention on empty sums}: For a sequence of real numbers $\{a_k\}_{k=1}^{\infty}$, we follow the convention that $\sum_{k=m}^{n} a_k = 0$ for any integers $m > n$.
\end{itemize}

We also collect the following technical lemmas that will be used in the subsequent convergence analysis.

\begin{lemma}\label{lemma-nonexpansive}\textup{(\cite[Corollary 23.11, Exercise 4.13]{bc2017Convex})}
If $\mathcal{A}$ is a maximally monotone operator and $\lambda > 0$, then the resolvent $\mathcal{J}_{\lambda\mathcal{A}}$ satisfies the following properties for all $\boldsymbol{x}, \boldsymbol{y} \in \mathbb{R}^n$:
\begin{enumerate}[label=\textit{(\roman*)}, leftmargin=2em]
    \item $\mathcal{J}_{\lambda\mathcal{A}}$ is firmly nonexpansive, i.e.,
        \[
        \big\|\mathcal{J}_{\lambda\mathcal{A}}\boldsymbol{x}-\mathcal{J}_{\lambda\mathcal{A}}\boldsymbol{y}\big\|^2+\big\|(\mathcal{I}-\mathcal{J}_{\lambda\mathcal{A}})\boldsymbol{x}-(\mathcal{I}-\mathcal{J}_{\lambda\mathcal{A}})\boldsymbol{y}\big\|^2\leq\|\boldsymbol{x}-\boldsymbol{y}\|^2.
        \]
    \item $\mathcal{R}_{\lambda\mathcal{A}} :=  \mathcal{I} - \mathcal{J}_{\lambda\mathcal{A}}$ is 1-co-coercive, i.e.,
        \[
        \big\|\mathcal{R}_{\lambda\mathcal{A}}\boldsymbol{x}-\mathcal{R}_{\lambda\mathcal{A}}\boldsymbol{y}\big\|^2
        \;\leq\;
        \langle\mathcal{R}_{\lambda\mathcal{A}}\boldsymbol{x}-\mathcal{R}_{\lambda\mathcal{A}}\boldsymbol{y},\, \boldsymbol{x}-\boldsymbol{y}\rangle.
        \]
    \item For any $\rho \in [0, 2]$, $\mathcal{T}_{\rho}$ is nonexpansive, i.e.,
        \[
        \big\|\mathcal{T}_{\rho}\boldsymbol{x} - \mathcal{T}_{\rho}\boldsymbol{y}\big\| \leq \|\boldsymbol{x} - \boldsymbol{y}\|.
        \]
\end{enumerate}
\end{lemma}

\begin{lemma}\label{lemma-sk-converge}\textup{(\cite[Lemma 2.5]{x2002Iterative})}
Let $\{s_n\}$ be a sequence of nonnegative real numbers satisfying
\begin{equation*}
s_{n+1} \leqslant (1 - \alpha_n)s_n + \alpha_n\gamma_n + \tau_n, \quad n \geqslant 0,
\end{equation*}
where $\{\alpha_n\}$, $\{\gamma_n\}$ and $\{\tau_n\}$ satisfy the conditions:
(i) $\{\alpha_n\} \subset [0,1]$, $\sum_{n=1}^{\infty} \alpha_n = \infty$;
(ii) $\limsup\limits_{n \rightarrow \infty}\,\gamma_n \leq 0$;
(iii) $\tau_n \geq 0~(n \geqslant 0)$, $\sum_{n=1}^{\infty} \tau_n < \infty$.
Then, $\lim\limits_{n\to\infty} s_n = 0$.
\end{lemma}

\noindent{\textbf{Semi-inner product and seminorm:}} {Let} $\mathcal{M} : \mathbb{R}^n \rightarrow \mathbb{R}^n$ {be} a linear, bounded, self-adjoint, and positive semidefinite operator. Then we define a semi-inner product $\langle v, w \rangle_{\mathcal{M}} := \langle v, \mathcal{M}w \rangle$ for all $v$ and $w$ in $\mathbb{R}^n$, as well as a continuous seminorm $\|w\|_{\mathcal{M}} = \sqrt{\langle w, w \rangle_{\mathcal{M}}}$.

Throughout this paper, we use $\|\cdot\|_2$ to denote the spectral norm of a linear operator, i.e., the operator norm induced by the Euclidean norm.

\medskip
\noindent{\textbf{Preconditioned PPA:}} {The} PPA with the preconditioner $\mathcal{M}$ \cite{bcln2022degenerate} for solving \eqref{pro-inclu} can be expressed as follows:
\begin{equation*}
    \boldsymbol{x}^0 \in \mathbb{R}^n, \quad \boldsymbol{x}^{k+1} = \mathcal{J}_{\mathcal{A}}^{\mathcal{M}}\boldsymbol{x}^k := (\mathcal{M} + \mathcal{A})^{-1}\mathcal{M}\boldsymbol{x}^k.
\end{equation*}
\begin{definition}[Admissible preconditioner \cite{bcln2022degenerate}]
An admissible preconditioner for the operator $\mathcal{A} : \mathbb{R}^n \rightrightarrows \mathbb{R}^n$ is a linear, bounded, self-adjoint, and positive semidefinite operator $\mathcal{M} : \mathbb{R}^n \rightarrow \mathbb{R}^n$ such that
\begin{equation*}
\mathcal{J}_{\mathcal{A}}^{\mathcal{M}} := (\mathcal{M} + \mathcal{A})^{-1}\mathcal{M}
\end{equation*}
is single-valued and has full domain.
\end{definition}

\begin{proposition}[Finite-dimensional version of Proposition 2.3 in \cite{bcln2022degenerate}]\label{M-decomposition}
Let $\mathcal{M}:\mathbb{R}^{n}\to\mathbb{R}^{n}$ be a linear, self-adjoint, and positive semidefinite operator. Set $m:=\operatorname{rank}\mathcal{M}$. Then, there exists an injective linear operator $\mathcal{C}:\mathbb{R}^{m}\to\mathbb{R}^{n}$ such that $\mathcal{M}=\mathcal{C}\mathcal{C}^{*}$. Its adjoint $\mathcal{C}^{*}:\mathbb{R}^{n}\to\mathbb{R}^{m}$ is surjective, and $\operatorname{ran}\mathcal{C}=\operatorname{ran}\mathcal{M}$. Moreover, $\|\mathcal{C}\|_{2}=\|\mathcal{C}^{*}\|_{2}=\sqrt{\|\mathcal{M}\|_{2}}$.
\end{proposition}

\begin{proposition}[Finite-dimensional version of Proposition 2.5 in \cite{syzz2024acc-pADMM}]\label{parallel-comp}
Let $\mathcal{A}: \mathbb{R}^n \rightrightarrows \mathbb{R}^n$ be a maximally monotone operator and let $\mathcal{M}$ be an admissible preconditioner for $\mathcal{A}$. Suppose that $\mathcal{M} = \mathcal{C}\mathcal{C}^*$ is a decomposition of $\mathcal{M}$ according to Proposition \ref{M-decomposition} with $\mathcal{C}: \mathbb{R}^{m} \rightarrow \mathbb{R}^n$. The parallel composition $\mathcal{C}^* \triangleright \mathcal{A} := (\mathcal{C}^* \mathcal{A}^{-1} \mathcal{C})^{-1}$ is a maximally monotone operator. Furthermore, the resolvent of $\mathcal{C}^* \triangleright \mathcal{A}$, denoted by $\mathcal{J}_{\mathcal{C}^* \triangleright \mathcal{A}} := (\mathcal{I} + \mathcal{C}^* \triangleright \mathcal{A})^{-1}$, satisfies the following identity
\begin{equation*}
\mathcal{J}_{\mathcal{C}^* \triangleright \mathcal{A}} = \mathcal{C}^* (\mathcal{M} + \mathcal{A})^{-1} \mathcal{C}.
\end{equation*}
In particular, $\mathcal{J}_{\mathcal{C}^* \triangleright \mathcal{A}}: \mathbb{R}^{m} \rightarrow \mathbb{R}^{m}$ is everywhere well-defined and firmly nonexpansive. Moreover, for any $\rho \in (0, 2]$, define
\begin{equation}\label{eq-Frho}
\Frho := (1 - \rho)\mathcal{I} + \rho\mathcal{J}_{\mathcal{C}^* \triangleright \mathcal{A}}.
\end{equation}
Then, $\Frho$ is nonexpansive and
\begin{equation*}
\mathcal{C}^* {\operatorname{zero}}(\mathcal{A})
= \mathcal{C}^* {\operatorname{Fix}}(\mathcal{J}_{\mathcal{A}}^{\mathcal{M}})
= {\operatorname{Fix}}(\mathcal{J}_{\mathcal{C}^* \triangleright \mathcal{A}})
= {\operatorname{Fix}}( \Frho),
\end{equation*}
where $\operatorname{zero}(\cdot)$ and $\operatorname{Fix}(\cdot)$ denote the zero set and fixed-point set of an operator, respectively. In particular, $\operatorname{zero}(\mathcal{A})=\mathcal{S}$.
\end{proposition}

\section{An inexact Halpern-accelerated preconditioned generalized PPA}\label{sec-iAccGPPA}

In this section, we present the \iHPGPPA for solving the inclusion problem \eqref{pro-inclu}. Its convergence properties are established in Section \ref{sec-convergence}. The complete algorithmic framework is presented in Algorithm \ref{algo-iAccGPPA}.

\begin{algorithm}[H]
\caption{\iHPGPPA for the inclusion problem \eqref{pro-inclu}}\label{algo-iAccGPPA}
\textbf{Input:} Choose $\rho\in(0,2]$, a sequence $\{\beta_k\}_{k=0}^{\infty}\subseteq(0,1)$, {an admissible preconditioner $\mathcal{M}$}, and an inexactness tolerance sequence $\{\varepsilon_k\}_{k=0}^{\infty}\subseteq{[0,\infty)}$. Select an initial point $\boldsymbol{x}^0\in\mathbb{R}^n$ and set $k=0$.

\textbf{while} a termination criterion is not met, \textbf{do}
\begin{subequations}\label{Hal-PPP-inex}
\begin{align}
\text{Find} &~~\boldsymbol{y}^k ~~\text{such that}~ \big\|\boldsymbol{y}^k-\mathcal{J}^{\mathcal{M}}_{\mathcal{A}}\boldsymbol{x}^k\big\| \leq \varepsilon_k, \label{inexact-criterion}\\
\boldsymbol{x}^{k+1} &= \beta_k \boldsymbol{x}^0 + (1 - \beta_k)\big[\rho \boldsymbol{y}^k+(1-\rho)\boldsymbol{x}^k\big], \quad k \geq 0. \label{step-Hal-relax}
\end{align}
\end{subequations}
Set $k=k+1$.
\vspace{1mm}

\textbf{end while}
\end{algorithm}

The \iHPGPPA framework allows inexact solutions to the resolvent subproblem with a prescribed inexactness tolerance $\varepsilon_k$, where ${\boldsymbol{y}}^k$ serves as an inexact evaluation of $\mathcal{J}_{\mathcal{A}}^{{\mathcal{M}}}\boldsymbol{x}^k$ satisfying \eqref{inexact-criterion}. The update step \eqref{step-Hal-relax} integrates two key techniques: relaxation and Halpern acceleration. Here $\beta_k$ controls the pull toward the anchor $\boldsymbol{x}^0$, while $\rho$ determines which relaxed resolvent mapping is being accelerated. Consequently, the parameter $\rho$ affects not only the practical behavior of the method but also the admissible level of inexactness in the convergence-rate analysis. In particular, the endpoint $\rho=2$ is qualitatively different from the interior regime $\rho\in(0,2)$, and this distinction will be reflected in the results below.

Let $\mathcal{M} = \mathcal{C}\mathcal{C}^*$ be a decomposition of $\mathcal{M}$ according to Proposition \ref{M-decomposition}. To prove the convergence of the scheme \eqref{Hal-PPP-inex}, we first introduce the following auxiliary sequences
\begin{equation}\label{auxi-seq}
\boldsymbol{u}^k = \mathcal{C}^*\boldsymbol{x}^k,  \quad \widetilde{\boldsymbol{u}}^k = \mathcal{C}^*\boldsymbol{y}^k, \quad \forall k\geq0.
\end{equation}
Then we have
\begin{equation*}
\mathcal{C}^*\mathcal{J}^{\mathcal{M}}_{\mathcal{A}}\boldsymbol{x}^k = \mathcal{C}^* (\mathcal{M} + \mathcal{A})^{-1} \mathcal{C}\boldsymbol{u}^k = \mathcal{J}_{\mathcal{C}^* \triangleright \mathcal{A}}\boldsymbol{u}^k.
\end{equation*}
Multiplying by $\mathcal{C}^*$ at each step of \eqref{Hal-PPP-inex}, we have
\begin{equation}\label{Hal-PPP-inex-u}
\begin{aligned}
&\big\|\widetilde{\boldsymbol{u}}^k-\mathcal{J}_{\mathcal{C}^* \triangleright \mathcal{A}}\boldsymbol{u}^k\big\|
= \big\|\mathcal{C}^*(\boldsymbol{y}^k-\mathcal{J}_{ \mathcal{A}}^{\mathcal{M}}\boldsymbol{x}^k)\big\|
\leq \|\mathcal{C}^*\|_2\,\varepsilon_k,
\\
& \boldsymbol{u}^{k+1} = \beta_k \boldsymbol{u}^0 + (1 - \beta_k)\big[\rho \widetilde{\boldsymbol{u}}^k+(1-\rho)\boldsymbol{u}^k\big], \quad k \geq 0.
\end{aligned}
\end{equation}
Moreover, by Proposition \ref{parallel-comp}, $\mathcal{C}^*\triangleright\mathcal{A}$ is maximally monotone and its zero set is $\mathcal{C}^*\mathcal{S}$. For ease of the subsequent analysis, define
\begin{equation*}
\mathcal{R}_{\triangleright}:=\mathcal{I}-\mathcal{J}_{\mathcal{C}^*\triangleright\mathcal{A}}.
\end{equation*}
By Lemma \ref{lemma-nonexpansive} applied to $\mathcal{C}^*\triangleright\mathcal{A}$ with $\lambda=1$, $\mathcal{R}_{\triangleright}$ is 1-co-coercive. We use the nonexpansive operator $\Frho$ defined in Proposition \ref{parallel-comp}.

\section{Convergence analysis}\label{sec-convergence}

\subsection{Convergence}

We next analyze the convergence of the iHP-GPPA. To this end, we first characterize the boundedness property of the sequence generated by {\eqref{Hal-PPP-inex-u}} in the following proposition.

\begin{proposition}\label{propo-bound}\textup{(\textbf{Boundedness property})}
Suppose that the solution set $\mathcal{S}$ is nonempty. Let $\{\boldsymbol{u}^k\}$ be the sequence generated by {\eqref{Hal-PPP-inex-u}} with $\rho\in(0,2]$ and {$\beta_k\in[0,1]$}. Then, for any {point} $\boldsymbol{u}^*\in\mathcal{C}^*\mathcal{S}$, it holds that
\begin{equation}\label{eq-propo-bound}
\big\|\boldsymbol{u}^k-\boldsymbol{u}^*\big\|\leq \big\|\boldsymbol{u}^0-\boldsymbol{u}^*\big\|+\rho\sum_{i=0}^{k-1}\prod_{j=i}^{k-1}(1-\beta_j) {\|\mathcal{C}^*\|_2}\,\varepsilon_i, \quad \forall\,k\geq0.
\end{equation}
Moreover, if $\sum_{k=0}^{\infty}\varepsilon_k<\infty$, then $\{\boldsymbol{u}^k\}$ is bounded.
\end{proposition}
\begin{proof}
We prove this proposition by induction. First, {\eqref{eq-propo-bound}} holds trivially for $k=0$. We now suppose that {\eqref{eq-propo-bound}} holds for $k=n$ for some integer $n\geq{0}$. Next, we show that {\eqref{eq-propo-bound}} also holds for $k=n+1$. Recall from Proposition \ref{parallel-comp} that the operator $\Frho$ defined in \eqref{eq-Frho} is nonexpansive for any $\rho\in(0,2]$. Moreover, it is known that any {point} $\boldsymbol{u}^*\in\mathcal{C}^*\mathcal{S}$ is a fixed point of $\Frho$. Using these facts and {\eqref{Hal-PPP-inex-u}}, we see that
\begin{equation*}
\begin{aligned}
&\big\|\boldsymbol{u}^{n+1}-\boldsymbol{u}^*\big\|
\;=\; \big\|\beta_n \boldsymbol{u}^0 + (1-\beta_n)\big[\rho\widetilde{\boldsymbol{u}}^n+(1-\rho)\boldsymbol{u}^n\big]-\boldsymbol{u}^*\big\| \\
&= \big\|\beta_n(\boldsymbol{u}^0-\boldsymbol{u}^*) + (1-\beta_n)\big[\rho \mathcal{J}_{\mathcal{C}^*\triangleright\mathcal{A}}\boldsymbol{u}^n+(1-\rho)\boldsymbol{u}^n-\boldsymbol{u}^*+\rho(\widetilde{\boldsymbol{u}}^n-\mathcal{J}_{\mathcal{C}^*\triangleright\mathcal{A}}\boldsymbol{u}^n)\big]\big\| \\
&\leq \beta_n\big\|\boldsymbol{u}^0-\boldsymbol{u}^*\big\|
+ (1-\beta_n)\big\|\Frho\boldsymbol{u}^n-\boldsymbol{u}^*\big\|
+ (1-\beta_n)\rho{\|\widetilde{\boldsymbol{u}}^n-\mathcal{J}_{\mathcal{C}^*\triangleright\mathcal{A}}\boldsymbol{u}^n\|}  \\
&\leq \beta_n\big\|\boldsymbol{u}^0-\boldsymbol{u}^*\big\| + (1-\beta_n)\big\|\boldsymbol{u}^n-\boldsymbol{u}^*\big\|
+ (1-\beta_n)\rho{\|\mathcal{C}^*\|_2\,\varepsilon_n}  \\
&\leq \beta_n\big\|\boldsymbol{u}^0-\boldsymbol{u}^*\big\| + (1-\beta_n)\Big(\big\|\boldsymbol{u}^0-\boldsymbol{u}^*\big\|+\rho\sum_{i=0}^{n-1}\prod_{j=i}^{n-1}(1-\beta_j)\|\mathcal{C}^*\|_2\,\varepsilon_i\Big) + (1-\beta_n)\rho\|\mathcal{C}^*\|_2\,\varepsilon_n\\
&= \big\|\boldsymbol{u}^0-\boldsymbol{u}^*\big\| + \rho\sum_{i=0}^{n-1}\prod_{j=i}^{n}(1-\beta_j)\|\mathcal{C}^*\|_2\,\varepsilon_i + (1-\beta_n)\rho\|\mathcal{C}^*\|_2\,\varepsilon_n\\
&= \big\|\boldsymbol{u}^0-\boldsymbol{u}^*\big\| + \rho\sum_{i=0}^{n}\prod_{j=i}^{n}(1-\beta_j)\|\mathcal{C}^*\|_2\,\varepsilon_i,
\end{aligned}
\end{equation*}
where the second inequality follows from the nonexpansivity of $\Frho$ and the inexactness criterion {\eqref{Hal-PPP-inex-u}}, and the third inequality follows from the induction hypothesis. Therefore, {\eqref{eq-propo-bound}} holds for $k=n+1$. This completes the induction. Moreover, if $\sum_{k=0}^{\infty}\varepsilon_k<\infty$, it readily follows from \eqref{eq-propo-bound} that the sequence $\{\boldsymbol{u}^k\}$ is bounded. This completes the proof.
\end{proof}

Then, based on the above boundedness property, the following lemma establishes the convergence of the fixed-point residual for \eqref{Hal-PPP-inex-u}.

\begin{lemma}\label{lemma-Rto0}
Suppose that the solution set $\mathcal{S}$ is nonempty. Let $\{\boldsymbol{u}^k\}$ be generated by {\eqref{Hal-PPP-inex-u}} with $\rho\in(0,2]$. If $\sum_{k=0}^{\infty}\varepsilon_k<\infty$ and the sequence $\{\beta_k\}_{k=0}^{\infty} \subset [0, 1]$ satisfies conditions: (i) $\lim\limits_{k\to\infty}\beta_k = 0$; (ii) $\sum_{k=0}^{\infty}\beta_k=\infty$; (iii) $\sum_{k=1}^{\infty}|\beta_k-\beta_{k-1}|<\infty$,
then it holds that
\begin{equation*}
\lim\limits_{k\to\infty}~\|\mathcal{R}_{\triangleright}\boldsymbol{u}^{k}\| = 0.
\end{equation*}
\end{lemma}
\begin{proof}
We first establish that $\lim_{k\to\infty}\|\boldsymbol{u}^{k+1}-\boldsymbol{u}^k\|=0$. By the update rule {\eqref{Hal-PPP-inex-u}}, we have
\begin{equation*}
\begin{aligned}
    &\big\|\boldsymbol{u}^{k+1}-\boldsymbol{u}^k\big\|
    \\
    =&\; \big\|(\beta_k-\beta_{k-1}) \boldsymbol{u}^0 + (1 - \beta_k)\big[\rho\widetilde{\boldsymbol{u}}^k+(1-\rho)\boldsymbol{u}^k\big] - (1 - \beta_{k-1})\big[\rho\widetilde{\boldsymbol{u}}^{k-1}+(1-\rho)\boldsymbol{u}^{k-1}\big]\big\|\\
    \leq&\;\big\|(\beta_k-\beta_{k-1}) \boldsymbol{u}^0 + (1 - \beta_k)\Frho\boldsymbol{u}^k - (1 - \beta_{k-1})\Frho\boldsymbol{u}^{k-1}\big\|+\rho(\|\mathcal{C}^*\|_2\,\varepsilon_k+
    \|\mathcal{C}^*\|_2\,\varepsilon_{k-1}),
\end{aligned}
\end{equation*}
where the inequality follows from
the {identity $\Frho \boldsymbol{u}^k+\rho(\widetilde{\boldsymbol{u}}^k- \mathcal{J}_{\mathcal{C}^*\triangleright\mathcal{A}} \boldsymbol{u}^k) = \rho \widetilde{\boldsymbol{u}}^k + (1-\rho)\boldsymbol{u}^k$}
 and the inexactness criterion {\eqref{Hal-PPP-inex-u}}.
By the triangle inequality, the first term on the right-hand side can be bounded as
\begin{equation*}
\begin{aligned}
&\big\|(\beta_k-\beta_{k-1}) \boldsymbol{u}^0 + (1 - \beta_k)\Frho\boldsymbol{u}^k - (1 - \beta_{k-1})\Frho\boldsymbol{u}^{k-1}\big\|\\
\leq& (1 - \beta_k)\big\|\Frho\boldsymbol{u}^k - \Frho\boldsymbol{u}^{k-1}\big\|+|\beta_k-\beta_{k-1}|\big\|\boldsymbol{u}^0 - \Frho\boldsymbol{u}^{k-1}\big\|\\
\leq& (1 - \beta_k)\big\|\boldsymbol{u}^k - \boldsymbol{u}^{k-1}\big\|+|\beta_k-\beta_{k-1}|\big\|\boldsymbol{u}^0 - \Frho\boldsymbol{u}^{k-1}\big\|,
\end{aligned}
\end{equation*}
where the last inequality follows from the nonexpansivity of $\Frho$. Substituting this into the previous inequality yields
\begin{equation}\label{ineq-x_dec}
    \big\|\boldsymbol{u}^{k+1}-\boldsymbol{u}^k\big\|\leq(1 - \beta_k)\big\|\boldsymbol{u}^k - \boldsymbol{u}^{k-1}\big\|+|\beta_k-\beta_{k-1}|\big\|\boldsymbol{u}^0 - \Frho\boldsymbol{u}^{k-1}\big\|+\rho(\|\mathcal{C}^*\|_2\,\varepsilon_k+\|\mathcal{C}^*\|_2\,\varepsilon_{k-1}).
\end{equation}
By Proposition \ref{propo-bound} and the nonexpansivity of $\Frho$, {there exists $B>0$ such that $\|\boldsymbol{u}^0-\Frho\boldsymbol{u}^{k-1}\|\leq B$}. We now apply Lemma \ref{lemma-sk-converge} to \eqref{ineq-x_dec}. Let $s_k=\|\boldsymbol{u}^{{k}}-\boldsymbol{u}^{{k-1}}\|$, $\gamma_k={0}$, and
$\tau_k=B|\beta_k-\beta_{k-1}|+\rho(\|\mathcal{C}^*\|_2\,\varepsilon_k+\|\mathcal{C}^*\|_2\,\varepsilon_{k-1})$. {Condition (iii) and the summability of $\{\varepsilon_k\}$ give $\sum_{k=1}^{\infty}\tau_k<\infty$.} Hence, Lemma \ref{lemma-sk-converge} gives $\lim_{k\to\infty}\|\boldsymbol{u}^{k+1}-\boldsymbol{u}^k\|=0$.

Next, we show $\lim_{k \to \infty}\|\mathcal{R}_{\triangleright}\boldsymbol{u}^k\| = 0$. By the definitions of $\mathcal{R}_{\triangleright}$ and $\Frho$, we have that
\begin{equation}\label{eq-R-T_rho}
  \rho\big\|\mathcal{R}_{\triangleright}\boldsymbol{u}^k \big\| =  \rho\big\|\boldsymbol{u}^k - \mathcal{J}_{\mathcal{C}^*\triangleright\mathcal{A}}\boldsymbol{u}^k\big\|
    = \big\|\boldsymbol{u}^k - \Frho\boldsymbol{u}^k\big\|
    \leq\big\|\boldsymbol{u}^{k+1} - \boldsymbol{u}^k\big\|+\big\|\boldsymbol{u}^{k+1} - \Frho\boldsymbol{u}^k\big\|.
\end{equation}
Substituting the update rule {\eqref{Hal-PPP-inex-u}} into $\|\boldsymbol{u}^{k+1} - \Frho\boldsymbol{u}^k\|$ and applying the inexactness criterion {\eqref{Hal-PPP-inex-u}}, we have that
\begin{equation*}
   \big\|\boldsymbol{u}^{k+1} - \Frho\boldsymbol{u}^k\big\|\leq\big\|\beta_k \boldsymbol{u}^0 + (1 - \beta_k)\big[\rho\widetilde{\boldsymbol{u}}^k+(1-\rho)\boldsymbol{u}^k\big] - \Frho\boldsymbol{u}^k\big\|\leq \beta_k\big\|\boldsymbol{u}^0 - \Frho\boldsymbol{u}^k\big\|+(1 - \beta_k)\rho\|\mathcal{C}^*\|_2\,\varepsilon_k.
\end{equation*}
Using the above inequality in \eqref{eq-R-T_rho}, we obtain that
\begin{equation*}
    \rho\|\mathcal{R}_{\triangleright}\boldsymbol{u}^k\| \leq \|\boldsymbol{u}^{k+1} - \boldsymbol{u}^k\| + \beta_k\|\boldsymbol{u}^0 - \Frho\boldsymbol{u}^k\| + (1 - \beta_k)\rho\|\mathcal{C}^*\|_2\,\varepsilon_k.
\end{equation*}
By Proposition \ref{propo-bound}, the sequence $\{\boldsymbol{u}^k\}$ is bounded. Since $\Frho$ is nonexpansive, $\|\boldsymbol{u}^0 - \Frho\boldsymbol{u}^k\|$ is bounded. Moreover, $\lim_{k \to \infty}\|\boldsymbol{u}^{k+1} - \boldsymbol{u}^k\| = 0$, $\lim_{k \to \infty}\beta_k = 0$, and $\lim_{k \to \infty}\varepsilon_k = 0$. Taking the limit as $k \to \infty$ on both sides, we get $\lim_{k \to \infty}\|\mathcal{R}_{\triangleright}\boldsymbol{u}^k\| = 0$. This completes the proof.
\end{proof}

With the above preparations, we are now ready to establish the convergence of the whole sequence generated by {\eqref{Hal-PPP-inex-u}}. Our convergence analysis is similar to the analytical approach used in \cite{x2002Iterative}, but we generalize the selection of the relaxation parameter $\rho$.

\begin{lemma}\label{lemma-aux-convergence}\textup{(\textbf{Global convergence})}
    Suppose that the solution set $\mathcal{S}$ is nonempty. Let $\{\boldsymbol{u}^k\}$ be generated by {\eqref{Hal-PPP-inex-u}} with $\rho\in(0,2]$ and $\beta_k$ satisfying the conditions in Lemma \ref{lemma-Rto0}. If $\sum_{k=0}^{\infty}\varepsilon_k<\infty$, then
    \begin{equation*}
        \lim_{k\to\infty} \boldsymbol{u}^k = \boldsymbol{q}:= P_{\mathcal{C}^*\mathcal{S}}(\boldsymbol{u}^0),
    \end{equation*}
    which implies that $\{\boldsymbol{u}^k\}_{k=0}^{\infty}$ converges to the {point in $\mathcal{C}^*\mathcal{S}$} that is closest to $\boldsymbol{u}^0$.
\end{lemma}
\begin{proof}
By Proposition \ref{parallel-comp}, $\mathcal{C}^{*}\mathcal{S}=\operatorname{Fix}(\Frho)$. This set is nonempty because $\mathcal{S}\neq\varnothing$, and is closed and convex in $\mathbb{R}^{m}$ because $\Frho$ is nonexpansive. Thus,  the projection $P_{\mathcal{C}^{*}\mathcal{S}}$ is well-defined. Next, we show $\{\boldsymbol{u}^k\}_{k=0}^{\infty}$ converges to $\boldsymbol{q}:=P_{\mathcal{C}^*\mathcal{S}}(\boldsymbol{u}^0)$. Take a subsequence $\{\boldsymbol{u}^{k_j}\}$ of $\{\boldsymbol{u}^{k}\}$ such that
\begin{equation}\label{eq-subseq}
\limsup_{k\to\infty} \langle \boldsymbol{u}^0 - \boldsymbol{q}, \boldsymbol{u}^k - \boldsymbol{q} \rangle = \lim_{j\to\infty} \langle \boldsymbol{u}^0 - \boldsymbol{q}, \boldsymbol{u}^{k_j} - \boldsymbol{q} \rangle.
\end{equation}
By Proposition \ref{propo-bound}, the sequence $\{\boldsymbol{u}^k\}_{k=0}^{\infty}$ is bounded, so the subsequence $\{\boldsymbol{u}^{k_j}\}_{j=0}^{\infty}$ is also bounded. By the Bolzano-Weierstrass theorem, and passing to a subsequence if necessary, we may assume $\lim_{j\to\infty}\boldsymbol{u}^{k_j}= \boldsymbol{u}^\infty$. Moreover, from Lemma \ref{lemma-Rto0}, we have that $\lim_{j\to\infty}\big\|\boldsymbol{u}^{k_j} - \mathcal{J}_{\mathcal{C}^*\triangleright\mathcal{A}}\boldsymbol{u}^{k_j}\big\|=0$. Since $\mathcal{J}_{\mathcal{C}^*\triangleright\mathcal{A}}$ is nonexpansive, it is continuous. Thus, taking the limit as $j \to \infty$ gives $\boldsymbol{u}^\infty = \mathcal{J}_{\mathcal{C}^*\triangleright\mathcal{A}}\boldsymbol{u}^\infty$ and hence $\boldsymbol{u}^\infty \in \mathcal{C}^*\mathcal{S}$. Combining \eqref{eq-subseq} with {\cite[Section 3]{gs1984uniform}}, we obtain
\begin{equation}\label{eq-subseq-leq}
\limsup_k \langle \boldsymbol{u}^0 - \boldsymbol{q}, \boldsymbol{u}^k - \boldsymbol{q} \rangle = \langle \boldsymbol{u}^0 - \boldsymbol{q}, \boldsymbol{u}^\infty - \boldsymbol{q} \rangle\leq 0.
\end{equation}

Next, we derive a recursive inequality for $\|\boldsymbol{u}^{k+1} - \boldsymbol{q}\|^2$. Using the update rule {\eqref{Hal-PPP-inex-u}} and the inequality $\|\boldsymbol{x}+\boldsymbol{y}\|^2\leq \|\boldsymbol{x}\|^2+2\langle \boldsymbol{y},\boldsymbol{x}+\boldsymbol{y}\rangle$, we have that
\begin{equation}\label{ineq-converge-q}
\begin{aligned}
\|\boldsymbol{u}^{k+1}-\boldsymbol{q}\|^2
&= \|\beta_k(\boldsymbol{u}^0-\boldsymbol{q})+(1-\beta_k)[\rho\widetilde{\boldsymbol{u}}^k+(1-\rho)\boldsymbol{u}^k-\boldsymbol{q}]\|^2\\
&\leq (1-\beta_k)\|\rho\widetilde{\boldsymbol{u}}^k+(1-\rho)\boldsymbol{u}^k-\boldsymbol{q}\|^2 + 2\beta_k\langle \boldsymbol{u}^0 - \boldsymbol{q}, \boldsymbol{u}^{k+1} - \boldsymbol{q} \rangle,
\end{aligned}
\end{equation}
Recalling the nonexpansivity of $\Frho$ defined in \eqref{eq-Frho}, we have that
\begin{align*}
    \,\|\rho\widetilde{\boldsymbol{u}}^k+(1-\rho)\boldsymbol{u}^k-\boldsymbol{q}\|^2
    \;\leq\; & (\|\rho\mathcal{J}_{\mathcal{C}^*\triangleright\mathcal{A}}\boldsymbol{u}^k+(1-\rho)\boldsymbol{u}^k-\boldsymbol{q}\|+\rho\|\mathcal{C}^*\|_2\varepsilon_k)^2
    \\
    \leq\; &(\|\boldsymbol{u}^k-\boldsymbol{q}\|+\rho\|\mathcal{C}^*\|_2\,\varepsilon_k)^2,
\end{align*}
where the first inequality follows from the inexactness criterion {\eqref{Hal-PPP-inex-u}}, and the second inequality follows from the nonexpansivity of $\Frho$ {and the fact that
$\boldsymbol{q}=\mathcal{J}_{\mathcal{C}^*\triangleright\mathcal{A}}\boldsymbol{q}$.} Substituting the expression above into \eqref{ineq-converge-q} yields
\begin{equation}\label{ineq-bound-q}
\begin{aligned}
& \|\boldsymbol{u}^{k+1}-\boldsymbol{q}\|^2
\leq(1-\beta_k)(\|\boldsymbol{u}^k-\boldsymbol{q}\|+\rho\|\mathcal{C}^*\|_2\,\varepsilon_k)^2+2\beta_k\langle \boldsymbol{u}^0 - \boldsymbol{q}, \boldsymbol{u}^{k+1} - \boldsymbol{q} \rangle\\
&\leq (1-\beta_k)\|\boldsymbol{u}^k-\boldsymbol{q}\|^2 +{2\rho\|\mathcal{C}^*\|_2\,\varepsilon_k(\rho\|\mathcal{C}^*\|_2\,\varepsilon_k + \|\boldsymbol{u}^k-\boldsymbol{q}\|)}
\;+\; 2\beta_k\langle \boldsymbol{u}^0 - \boldsymbol{q}, \boldsymbol{u}^{k+1} - \boldsymbol{q} \rangle.
\end{aligned}
\end{equation}
Let $\tau_k = 2\rho\|\mathcal{C}^*\|_2\,\varepsilon_k\left(\rho\|\mathcal{C}^*\|_2\,\varepsilon_k + \|\boldsymbol{u}^k - \boldsymbol{q}\|\right)$. By Proposition \ref{propo-bound}, $\|\boldsymbol{u}^k - \boldsymbol{q}\|$ is bounded, and since $\sum_{k=0}^{\infty}\varepsilon_k < \infty$, we have $\sum_{k=0}^{\infty}\tau_k < \infty$.
We now apply Lemma \ref{lemma-sk-converge} to the recursive inequality \eqref{ineq-bound-q}. Let $s_k = \|\boldsymbol{u}^k - \boldsymbol{q}\|^2$, $\gamma_k = 2\langle \boldsymbol{u}^0 - \boldsymbol{q}, \boldsymbol{u}^{k+1} - \boldsymbol{q} \rangle$ with $\limsup_{k \to \infty}\gamma_k \leq 0$ by \eqref{eq-subseq-leq}, and $\tau_k$ as defined above. By Lemma \ref{lemma-sk-converge}, we conclude that $\lim_{k \to \infty} s_k = 0$, which implies $\lim_{k \to \infty}\|\boldsymbol{u}^k - \boldsymbol{q}\| = 0$. This completes the proof.
\end{proof}

Based on the auxiliary sequences in \eqref{auxi-seq}, we can establish the convergence properties of scheme \eqref{Hal-PPP-inex} in the following theorem.

\begin{theorem}\label{thm-preconditioned-convergence}\textup{(\textbf{Global convergence})}
Suppose that $\mathcal{M}$ is an admissible preconditioner and the solution set $\mathcal{S}$ is nonempty. Let $\mathcal{M}=\mathcal{C}\mathcal{C}^*$ be a decomposition of $\mathcal{M}$, $\{\boldsymbol{x}^k\}$ and $\{\boldsymbol{y}^k\}$ be the sequences generated by \eqref{Hal-PPP-inex} with {$\rho\in(0,2]$ and $\{\beta_k\}\subset(0,1)$ satisfying the conditions in Lemma \ref{lemma-Rto0}}. Suppose that $\sum_{k=0}^{\infty}\varepsilon_k<\infty$. Then, the sequence $\{\boldsymbol{y}^k\}$ converges to a solution $\boldsymbol{x}^*\in\mathcal{S}${, which is the unique minimizer of $\|\boldsymbol{x}^0-\boldsymbol{x}\|_{\mathcal{M}}$ over $\boldsymbol{x}\in\mathcal{S}$. Moreover, if $0<\rho<2$, then $\boldsymbol{x}^k\to\boldsymbol{x}^*$.}
\end{theorem}
\begin{proof}
It follows from {Lemma \ref{lemma-aux-convergence}} that {$\boldsymbol{u}^k\to\boldsymbol{u}^*:=\boldsymbol{q}\in\mathcal{C}^*\mathcal{S}$}.
Moreover, $\boldsymbol{u}^*$ solves $0 \in (\mathcal{C}^* \triangleright \mathcal{A})\boldsymbol{u}$ if and only if $\boldsymbol{u}^*=\mathcal{J}_{\mathcal{C}^* \triangleright \mathcal{A}}\boldsymbol{u}^*$.
Thus $\boldsymbol{u}^* \in {\operatorname{Fix}}( \mathcal{J}_{\mathcal{C}^* \triangleright \mathcal{A}})$. By utilizing Proposition \ref{parallel-comp}, we have
\[\mathcal{C}^* {\operatorname{zero}}(\mathcal{A}) = \mathcal{C}^* \, {\operatorname{Fix}} ({\mathcal{J}}_{\mathcal{A}}^{\mathcal{M}}) = {\operatorname{Fix}} (\mathcal{J}_{\mathcal{C}^* \triangleright \mathcal{A}}).
\]
Let $\boldsymbol{x}^* = (\mathcal{M} + \mathcal{A})^{-1} \mathcal{C} \boldsymbol{u}^*$. {Proposition \ref{parallel-comp} gives $\mathcal{C}^*\boldsymbol{x}^*=\mathcal{J}_{\mathcal{C}^*\triangleright\mathcal{A}}\boldsymbol{u}^*=
\boldsymbol{u}^*$.}

{Let $\mathcal{B}:=(\mathcal{M}+\mathcal{A})^{-1}$. Since $\mathcal{M}$ is monotone, continuous, and defined on all of $\mathbb{R}^n$, $\mathcal{M}+\mathcal{A}$ and hence $\mathcal{B}$ are maximally monotone. Admissibility implies that $\mathcal{B}$ is single-valued on $\operatorname{ran}\mathcal{M}=\operatorname{ran}\mathcal{C}$; in particular, $\mathcal{B}(\mathcal{C}\boldsymbol{u}^*)=\{\boldsymbol{x}^*\}$. By \cite[Corollary 1.3(4)]{aa1999geometrical}, $\mathcal{B}$ is continuous at $\mathcal{C}\boldsymbol{u}^*$, and thus $\boldsymbol{u}\mapsto\mathcal{B}(\mathcal{C}\boldsymbol{u})$ is continuous at $\boldsymbol{u}^*$.}

We now apply this continuity to $\boldsymbol{u}^k\to\boldsymbol{u}^*$. By the relationship $\boldsymbol{u}^k = \mathcal{C}^*\boldsymbol{x}^k$, we obtain
\begin{equation}\label{eq-y-convergence-bound}
\|\boldsymbol{y}^k-\boldsymbol{x}^*\|\leq \big\|{\mathcal{J}}_{\mathcal{A}}^{\mathcal{M}}\boldsymbol{x}^k-\boldsymbol{x}^*\big\|+\varepsilon_k = \|(\mathcal{M} + \mathcal{A})^{-1} \mathcal{C}\boldsymbol{u}^k - (\mathcal{M} + \mathcal{A})^{-1} \mathcal{C} \boldsymbol{u}^*\| + \varepsilon_k,
\end{equation}
where the inequality follows from $\big\|\boldsymbol{y}^k-\mathcal{J}^{\mathcal{M}}_{\mathcal{A}}\boldsymbol{x}^k\big\| \leq \varepsilon_k$.
The continuity established above and $\boldsymbol{u}^k\to\boldsymbol{u}^*$ imply that the first term on the right-hand side tends to zero. Since $\varepsilon_k\to0$, we conclude that
\begin{equation*}
\lim_{k\to\infty}\|\boldsymbol{y}^k-\boldsymbol{x}^*\| = 0.
\end{equation*}
Furthermore, we have
\begin{equation*}
\begin{aligned}
\mathcal{J}_{\mathcal{A}}^{\mathcal{M}}\boldsymbol{x}^* &= (\mathcal{M} + \mathcal{A})^{-1}\mathcal{C}\mathcal{C}^*\boldsymbol{x}^*
=(\mathcal{M} + \mathcal{A})^{-1}\mathcal{C}\boldsymbol{u}^* = \boldsymbol{x}^*,
\end{aligned}
\end{equation*}
which implies that $\boldsymbol{x}^*\in
\operatorname{Fix}(\mathcal{J}_{\mathcal{A}}^{\mathcal{M}}) = \operatorname{zero}(\mathcal{A})$ by Proposition \ref{parallel-comp}.

Since $\boldsymbol{u}^*=P_{\mathcal{C}^*\mathcal{S}}(\boldsymbol{u}^0)$, for any $\boldsymbol{s}\in\mathcal{S}$,
\[
 \|\boldsymbol{x}^0-\boldsymbol{x}^*\|_{\mathcal{M}}
 =\|\boldsymbol{u}^0-\boldsymbol{u}^*\|
 \leq\|\boldsymbol{u}^0-\mathcal{C}^*\boldsymbol{s}\|
 =\|\boldsymbol{x}^0-\boldsymbol{s}\|_{\mathcal{M}}.
\]
If $\boldsymbol{s}$ is another minimizer, uniqueness of the projection gives $\mathcal{C}^*\boldsymbol{s}=\mathcal{C}^*\boldsymbol{x}^*$, and hence $\mathcal{M}\boldsymbol{s}=\mathcal{M}\boldsymbol{x}^*$. By admissibility,
$\boldsymbol{s}={\mathcal{J}}_{\mathcal{A}}^{\mathcal{M}}\boldsymbol{s}={\mathcal{J}}_{\mathcal{A}}^{\mathcal{M}}\boldsymbol{x}^*=\boldsymbol{x}^*$.

For $0<\rho<2$, the update \eqref{step-Hal-relax} gives
\begin{equation}\label{eq-x-convergence-recursion}
\|\boldsymbol{x}^{k+1}-\boldsymbol{x}^*\|
\leq |1-\rho|\|\boldsymbol{x}^{k}-\boldsymbol{x}^*\|
+\beta_k\|\boldsymbol{x}^0-\boldsymbol{x}^*\|
+\rho\|\boldsymbol{y}^k-\boldsymbol{x}^*\|.
\end{equation}
Since $|1-\rho|<1$, $\beta_k\to0$, and $\boldsymbol{y}^k\to\boldsymbol{x}^*$, Lemma \ref{lemma-sk-converge}, applied with the coefficients {$\alpha_k = 1-|1-\rho|$, $\gamma_k = \alpha_k^{-1} (\beta_k\|\boldsymbol{x}^0-\boldsymbol{x}^*\|+\rho\|\boldsymbol{y}^k-\boldsymbol{x}^*\|)$, $\tau_k=0$,} gives $\boldsymbol{x}^k\to\boldsymbol{x}^*$.
\hfill
\end{proof}

Theorem \ref{thm-preconditioned-convergence} extends the convergence results in \cite[Theorem 2.7(a)]{syzz2024acc-pADMM} from exact subproblems with $\beta_k=\frac{1}{k+2}$ to inexact subproblems with more general anchoring parameters. Based on the above convergence result, a summable inexactness tolerance sequence $\{\varepsilon_k\}_{k=0}^{\infty}$ is sufficient to satisfy the inexactness condition in Theorem \ref{thm-preconditioned-convergence}. A canonical choice in practice is to set $\varepsilon_k = \mathcal{O}\left(\frac{1}{k^{1+\delta}}\right)$ for some constant $\delta > 0$, as this decay rate guarantees summability of the sequence. Moreover, the anchoring parameter sequence $\{\beta_k\}_{k=0}^{\infty}$ {can be chosen as} $\beta_k = {\frac{1}{k+r}}$ {with $r\geq2$}, {for which} all the conditions imposed on $\{\beta_k\}$ in Lemma \ref{lemma-Rto0} are fulfilled, as required for convergence.

\begin{corollary}[Positive definite $\mathcal{M}$]\label{rem-positive-definite}
Under the conditions on $\{\beta_k\}$ and $\{\varepsilon_k\}$ in Theorem \ref{thm-preconditioned-convergence}, suppose that $\mathcal{M}$ is positive definite and the solution set $\mathcal{S}$ is nonempty. Let $\{\boldsymbol{x}^k\}$ be generated by \eqref{Hal-PPP-inex} with $\rho\in(0,2]$. Then $\boldsymbol{x}^k\to\boldsymbol{x}^*$, where $\boldsymbol{x}^*$ is the unique minimizer of $\|\boldsymbol{x}^0-\boldsymbol{x}\|_{\mathcal{M}}$ over $\boldsymbol{x}\in\mathcal{S}$.
\end{corollary}
\begin{proof}
Since $\mathcal{M}$ is positive definite, it is admissible, so Theorem \ref{thm-preconditioned-convergence} applies. By Lemma \ref{lemma-aux-convergence} and the proof of Theorem \ref{thm-preconditioned-convergence}, we have that $\|\boldsymbol{x}^k-\boldsymbol{x}^*\|_{\mathcal{M}}
=\|\boldsymbol{u}^k-\boldsymbol{u}^*\|\to0$. Since $\|\cdot\|_{\mathcal{M}}$ is a norm equivalent to $\|\cdot\|$, we have $\boldsymbol{x}^k\to\boldsymbol{x}^*$ for all $\rho\in(0,2]$.
\hfill
\end{proof}

The following corollary relaxes the summability requirement on the inexactness tolerances and allows certain nonsummable tolerance sequences.

\begin{corollary}[Convergence with nonsummable tolerances]\label{cor-vanishing-errors}
The conclusions of Theorem \ref{thm-preconditioned-convergence} and Corollary \ref{rem-positive-definite} remain valid if $\sum_k\varepsilon_k<\infty$ is replaced by
\begin{equation}\label{eq-vanishing-errors}
 E_k:=\sum_{i=0}^{k-1}\varepsilon_i\prod_{j=i}^{k-1}(1-\beta_j)\longrightarrow0.
\end{equation}
This condition holds if either $\sum_k\varepsilon_k<\infty$ or $\varepsilon_k=o(\beta_k)$.
\end{corollary}
\begin{proof}
Let $\boldsymbol{v}^0=\boldsymbol{u}^0$ and
$\boldsymbol{v}^{k+1}=\beta_k\boldsymbol{u}^0+(1-\beta_k)\Frho\boldsymbol{v}^k$.
Nonexpansivity and \eqref{Hal-PPP-inex-u} give
\[
 \|\boldsymbol{u}^{k+1}-\boldsymbol{v}^{k+1}\|
 \leq(1-\beta_k)\big(\|\boldsymbol{u}^k-\boldsymbol{v}^k\|
       +\rho\|\mathcal{C}^*\|_2\varepsilon_k\big),
 \qquad
 \|\boldsymbol{u}^k-\boldsymbol{v}^k\|\leq\rho\|\mathcal{C}^*\|_2E_k.
\]
{Applying Lemma \ref{lemma-aux-convergence} to the exact sequence $\{\boldsymbol{v}^k\}$ yields $\boldsymbol{v}^k\to\boldsymbol{u}^*:=P_{\mathcal{C}^*\mathcal{S}}(\boldsymbol{u}^0)$. Thus, the above bound and $E_k\to0$ give $\boldsymbol{u}^k\to\boldsymbol{u}^*$. Also, $E_{k+1}\geq(1-\beta_k)\varepsilon_k$ and $\beta_k\to0$ imply $\varepsilon_k\to0$. Hence, \eqref{eq-y-convergence-bound} and the continuity established in the proof of Theorem \ref{thm-preconditioned-convergence} give $\boldsymbol{y}^k\to\boldsymbol{x}^*$. For $0<\rho<2$, \eqref{eq-x-convergence-recursion} and Lemma \ref{lemma-sk-converge} then yield $\boldsymbol{x}^k\to\boldsymbol{x}^*$. When $\mathcal{M}\succ0$, norm equivalence gives this convergence also for $\rho=2$, as in Corollary \ref{rem-positive-definite}.}

Finally, we verify the two sufficient conditions for $E_k\to0$. From its definition, $E_{k+1}=(1-\beta_k)E_k+(1-\beta_k)\varepsilon_k$.
Lemma \ref{lemma-sk-converge} proves $E_k\to0$ with
$\gamma_k=0$, $\tau_k=(1-\beta_k)\varepsilon_k$ in the summable case, or with
$\gamma_k=(1-\beta_k)\varepsilon_k/\beta_k\to0$, $\tau_k=0$ in the second case.
\hfill
\end{proof}

For $\beta_k=1/(k+r)$, condition \eqref{eq-vanishing-errors} becomes
$E_k=\sum_{i=0}^{k-1}(i+r-1)\varepsilon_i/(k+r-1)\to0$. Thus convergence holds whenever $\varepsilon_k=o(1/k)$, including the nonsummable choice $\varepsilon_k=1/((k+2)\log(k+2))$. This condition alone does not assert an accelerated residual rate.

The convergence of $\{\boldsymbol{x}^k\}$ in Theorem \ref{thm-preconditioned-convergence} cannot in general be extended to $\rho=2$ when $\mathcal{M}$ is singular, even with exact subproblems.

\begin{example}[Nonconvergence of $\{\boldsymbol{x}^k\}$ at $\rho=2$ when $\mathcal{M}$ is singular]\label{rem-singular-endpoint}
Take
\[
 \mathcal{M}=\begin{pmatrix}1&0\\0&0\end{pmatrix},
 \qquad
 \mathcal{A}=\begin{pmatrix}0&1\\-1&0\end{pmatrix}.
\]
The linear operator $\mathcal{A}$ is maximally monotone and $\mathcal{S}=\{\boldsymbol{0}\}$. Moreover,
\[
 (\mathcal{M}+\mathcal{A})^{-1}
 =\begin{pmatrix}0&-1\\1&1\end{pmatrix},
 \qquad
 {\mathcal{J}}_{\mathcal{A}}^{\mathcal{M}}
 \begin{pmatrix}
    x_1 \\ x_2
 \end{pmatrix}=\begin{pmatrix}
     0 \\ x_1
 \end{pmatrix}.
\]
Thus $\mathcal{M}$ is admissible and $(\mathcal{M}+\mathcal{A})^{-1}$ is Lipschitz. Set $\rho=2$, $\beta_k=\frac{1}{k+2}$, $\varepsilon_k=0$, and $\boldsymbol{x}^0=(1,0)^\top$. Direct substitution in \eqref{Hal-PPP-inex} gives, for $m\geq0$,
\[
 \boldsymbol{x}^{2m}
 =\left(\frac{1}{2m+1},-\frac{2m}{2m+1}\right)^\top,
 \qquad
 \boldsymbol{x}^{2m+1}=(0,1)^\top.
\]
Hence $\boldsymbol{y}^k={\mathcal{J}}_{\mathcal{A}}^{\mathcal{M}}\boldsymbol{x}^k\to\boldsymbol{0}$, but $\boldsymbol{x}^{2m}\to(0,-1)^\top$ and $\boldsymbol{x}^{2m+1}=(0,1)^\top$. In particular, the Euclidean residual $\|\boldsymbol{x}^k-{\mathcal{J}}_{\mathcal{A}}^{\mathcal{M}}\boldsymbol{x}^k\|=\|\boldsymbol{x}^k-\boldsymbol{y}^k\|$ does not tend to zero.
\end{example}

\subsection{Convergence rate}

In this subsection, we analyze the convergence rate of iHP-GPPA with $\beta_k=1/(k+r)$ and $r\geq2$. To this end, we first state a useful lemma for {\eqref{Hal-PPP-inex-u}}.

\begin{lemma}\label{lemma-z}
Let $\{\boldsymbol{u}^k\}$ be the sequence generated by {\eqref{Hal-PPP-inex-u}} with $\rho\in (0,2]$ and $\beta_k\in(0,1)$. Then, for $k \geq 0$, it holds that
\begin{enumerate}
\item[(i)] $\boldsymbol{u}^{k+1} - \boldsymbol{u}^k = \beta_k(\boldsymbol{u}^0 - \boldsymbol{u}^k) - \rho(1-\beta_k) (\boldsymbol{u}^k-\widetilde{\boldsymbol{u}}^k)$;
\item[(ii)] $\boldsymbol{u}^{k+1} - \boldsymbol{u}^k = \frac{\beta_k}{1 - \beta_k}(\boldsymbol{u}^0 - \boldsymbol{u}^{k+1}) - \rho (\boldsymbol{u}^k-\widetilde{\boldsymbol{u}}^k)$.
\end{enumerate}
\end{lemma}
\begin{proof}
The first relation is straightforward. For the second relation, we first observe that
\begin{equation*}
\boldsymbol{u}^0 - \boldsymbol{u}^{k+1} = (1 - \beta_k)(\boldsymbol{u}^0 - \boldsymbol{u}^k) + \rho(1 - \beta_k) (\boldsymbol{u}^k-\widetilde{\boldsymbol{u}}^k).
\end{equation*}
Rearranging the terms, we obtain
\begin{equation*}
\boldsymbol{u}^0 - \boldsymbol{u}^k = \frac{1}{1 - \beta_k}(\boldsymbol{u}^0 - \boldsymbol{u}^{k+1}) - \rho (\boldsymbol{u}^k-\widetilde{\boldsymbol{u}}^k).
\end{equation*}
Substituting this expression into the relation in statement (i) yields the desired result.
\hfill
\end{proof}

Then we characterize the descent property of an auxiliary sequence in the following {lemma}.

\begin{lemma}\label{lemma-suff-dec}\textup{(\textbf{Descent property})}
Let $\{\boldsymbol{u}^k\}$ be the sequence generated by {\eqref{Hal-PPP-inex-u}} with $\rho\in (0,2]$ and $\beta_k\in(0,1)$. Then, for all $k\geq0$ and $r \geq 2$, the following statements hold.
\begin{itemize}
    \item[(i)] For $\rho\in(0,2)$, it holds that
        \begin{equation}\label{ineq-any-beta}
        \begin{aligned}
        &\quad {\textstyle\frac{\rho}{2}}\big\|\mathcal{R}_{\triangleright}\boldsymbol{u}^{k+1}\big\|^2
        - {\textstyle\frac{\beta_k}{1-\beta_k}}\langle\mathcal{R}_{\triangleright}\boldsymbol{u}^{k+1}, \,\boldsymbol{u}^0-\boldsymbol{u}^{k+1} \rangle
        \\
        &
        \leq {\textstyle\frac{\rho(1-2\beta_k)}{2}}\big\|\mathcal{R}_{\triangleright}\boldsymbol{u}^{k}\big\|^2
        - \beta_k \langle \mathcal{R}_{\triangleright}\boldsymbol{u}^{k}, \,\boldsymbol{u}^0-\boldsymbol{u}^k \rangle  \\
        &\quad + \rho\beta_k\langle\mathcal{R}_{\triangleright}\boldsymbol{u}^{k}, \,\widetilde{\boldsymbol{u}}^k-\mathcal{J}_{\mathcal{C}^*\triangleright\mathcal{A}}\boldsymbol{u}^k \rangle
        + {\textstyle\frac{\rho^2}{2(2-\rho)}}\big\|\widetilde{\boldsymbol{u}}^k-\mathcal{J}_{\mathcal{C}^*\triangleright\mathcal{A}}\boldsymbol{u}^k \big\|^2.
        \end{aligned}
        \end{equation}
        Moreover, if $\beta_k=\frac{1}{ k+r}$ for all $k\geq 0$, it holds that
        \begin{equation}\label{ineq-suff-dec}
        \mathcal{L}_{k+1} \leq \mathcal{L}_k + {\textstyle\frac{\rho}{12}} \big\|\mathcal{R}_{\triangleright}\boldsymbol{u}^{k}\big\|^2 + {\textstyle\frac{2\rho(3-\rho)( k+r-1)^2}{(2-\rho)}}
        \big\|\widetilde{\boldsymbol{u}}^k - \mathcal{J}_{\mathcal{C}^*\triangleright\mathcal{A}}\boldsymbol{u}^k\big\|^2,
        \end{equation}
        where $\mathcal{L}_k$ is defined as
        \begin{equation}\label{def-Lk}
        \mathcal{L}_k
        := {\textstyle\frac{\rho ( k+r-2)( k+r-1)}{2}}\big\|\mathcal{R}_{\triangleright}\boldsymbol{u}^{k}\big\|^2
        - ( k+r-1)\langle \mathcal{R}_{\triangleright}\boldsymbol{u}^{k}, \,\boldsymbol{u}^0 - \boldsymbol{u}^k \rangle,
        \quad \forall\,k \geq 0.
        \end{equation}

    \item[(ii)] {If $\rho=2$, $\beta_k=1/(k+r)$, and $\widetilde{\boldsymbol{u}}^k=\mathcal{J}_{\mathcal{C}^*\triangleright\mathcal{A}}\boldsymbol{u}^k$ for all $k\geq0$, then the sequence in \eqref{def-Lk}, with $\rho=2$, satisfies}
        \begin{equation}\label{ineq-suff-dec-rho2}
        {\mathcal{L}_{k+1}\leq\mathcal{L}_k.}
        \end{equation}
\end{itemize}
\end{lemma}
\begin{proof}
First, it follows from Lemma \ref{lemma-nonexpansive}(ii) that
\begin{equation*}
\big\|\mathcal{R}_{\triangleright}\boldsymbol{u}^{k+1} - \mathcal{R}_{\triangleright}\boldsymbol{u}^{k}\big\|^2
\leq \langle \mathcal{R}_{\triangleright}\boldsymbol{u}^{k+1} - \mathcal{R}_{\triangleright}\boldsymbol{u}^{k}, \,\boldsymbol{u}^{k+1} - \boldsymbol{u}^k \rangle,
\end{equation*}
which implies that
\begin{equation}\label{ineq-coercive}
\begin{aligned}
&\quad {\textstyle\frac{\rho}{2}}\big\|\mathcal{R}_{\triangleright}\boldsymbol{u}^{k+1}\big\|^2
+ {\textstyle\frac{\rho}{2}}\big\|\mathcal{R}_{\triangleright}\boldsymbol{u}^{k}\big\|^2
+ {\textstyle\frac{2-\rho}{2}}\big\|\mathcal{R}_{\triangleright}\boldsymbol{u}^{k+1}-\mathcal{R}_{\triangleright}\boldsymbol{u}^{k}\big\|^2\\
&\leq \langle \mathcal{R}_{\triangleright}\boldsymbol{u}^{k+1}, \,\boldsymbol{u}^{k+1} - \boldsymbol{u}^k
+ \rho \mathcal{R}_{\triangleright}\boldsymbol{u}^{k} \rangle
- \langle \mathcal{R}_{\triangleright}\boldsymbol{u}^{k}, \,\boldsymbol{u}^{k+1} - \boldsymbol{u}^k \rangle.
\end{aligned}
\end{equation}
Then, recalling two expressions for $\boldsymbol{u}^{k+1}-\boldsymbol{u}^k$ in Lemma \ref{lemma-z} and the identity $\mathcal{J}_{\mathcal{C}^*\triangleright\mathcal{A}}\boldsymbol{u}^k+\mathcal{R}_{\triangleright}\boldsymbol{u}^k=\boldsymbol{u}^k$, we have that
\begin{equation*}
\begin{aligned}
\langle \mathcal{R}_{\triangleright}\boldsymbol{u}^{k+1}, \,\boldsymbol{u}^{k+1} - \boldsymbol{u}^k
+ \rho \mathcal{R}_{\triangleright}\boldsymbol{u}^{k} \rangle
&= \langle \mathcal{R}_{\triangleright}\boldsymbol{u}^{k+1}, \,{\textstyle\frac{\beta_k}{1-\beta_k}}(\boldsymbol{u}^0 - \boldsymbol{u}^{k+1}) - \rho (\boldsymbol{u}^k-\widetilde{\boldsymbol{u}}^k) + \rho \mathcal{R}_{\triangleright}\boldsymbol{u}^{k}\rangle \\
&= {\textstyle\frac{\beta_k}{1-\beta_k}}\langle \mathcal{R}_{\triangleright}\boldsymbol{u}^{k+1}, \,\boldsymbol{u}^0 - \boldsymbol{u}^{k+1}\rangle + \rho\langle \mathcal{R}_{\triangleright}\boldsymbol{u}^{k+1}, \,\widetilde{\boldsymbol{u}}^k - \mathcal{J}_{\mathcal{C}^*\triangleright\mathcal{A}}\boldsymbol{u}^k \rangle,
\end{aligned}
\end{equation*}
and
\begin{equation*}
\begin{aligned}
&\langle \mathcal{R}_{\triangleright}\boldsymbol{u}^{k}, \,\boldsymbol{u}^{k+1} - \boldsymbol{u}^k \rangle
\;=\; \langle \mathcal{R}_{\triangleright}\boldsymbol{u}^{k}, \,\beta_k (\boldsymbol{u}^0 - \boldsymbol{u}^k)
- \rho(1-\beta_k)(\boldsymbol{u}^k-\widetilde{\boldsymbol{u}}^k) \rangle  \\
&= \beta_k\langle\mathcal{R}_{\triangleright}\boldsymbol{u}^{k}, \,\boldsymbol{u}^0-\boldsymbol{u}^k\rangle
- \rho(1-\beta_k)\big(\langle\mathcal{R}_{\triangleright}\boldsymbol{u}^{k},\,\boldsymbol{u}^k-\widetilde{\boldsymbol{u}}^k-\mathcal{R}_{\triangleright}\boldsymbol{u}^{k}\rangle
+ \|\mathcal{R}_{\triangleright}\boldsymbol{u}^k\|^2\big) \\
&= \beta_k\langle\mathcal{R}_{\triangleright}\boldsymbol{u}^{k}, \,\boldsymbol{u}^0-\boldsymbol{u}^k\rangle
+ \rho(1-\beta_k)\langle\mathcal{R}_{\triangleright}\boldsymbol{u}^{k},\,\widetilde{\boldsymbol{u}}^k
-\mathcal{J}_{\mathcal{C}^*\triangleright\mathcal{A}}\boldsymbol{u}^k \rangle
- \rho(1-\beta_k)\|\mathcal{R}_{\triangleright}\boldsymbol{u}^k\|^2.
\end{aligned}
\end{equation*}
Substituting the two expressions above into \eqref{ineq-coercive} yields
\begin{equation}\label{ineq-R}
\begin{aligned}
&\quad {\textstyle\frac{\rho}{2}}\big\|\mathcal{R}_{\triangleright}\boldsymbol{u}^{k+1}\big\|^2
+ {\textstyle\frac{\rho}{2}}\big\|\mathcal{R}_{\triangleright}\boldsymbol{u}^{k}\big\|^2
+ {\textstyle\frac{2-\rho}{2}}\big\|\mathcal{R}_{\triangleright}\boldsymbol{u}^{k+1}-\mathcal{R}_{\triangleright}\boldsymbol{u}^{k}\big\|^2 \\
&\leq {\textstyle\frac{\beta_k}{1-\beta_k}} \langle \mathcal{R}_{\triangleright}\boldsymbol{u}^{k+1}, \boldsymbol{u}^0 - \boldsymbol{u}^{k+1} \rangle - \beta_k \langle \mathcal{R}_{\triangleright}\boldsymbol{u}^{k}, \boldsymbol{u}^0 - \boldsymbol{u}^k \rangle
+ \rho(1-\beta_k) \big\| \mathcal{R}_{\triangleright}\boldsymbol{u}^{k} \big\|^2  \\
&\;\quad
+ \rho \langle \mathcal{R}_{\triangleright}\boldsymbol{u}^{k+1} - \mathcal{R}_{\triangleright}\boldsymbol{u}^{k}, \,\widetilde{\boldsymbol{u}}^k - \mathcal{J}_{\mathcal{C}^*\triangleright\mathcal{A}}\boldsymbol{u}^k \rangle
+ \rho\beta_k \langle \mathcal{R}_{\triangleright}\boldsymbol{u}^{k}, \,\widetilde{\boldsymbol{u}}^k
-\mathcal{J}_{\mathcal{C}^*\triangleright\mathcal{A}}\boldsymbol{u}^k \rangle.
\end{aligned}
\end{equation}
We now consider two cases.

\textit{Case (i)}. Suppose $\rho\in(0,2)$. It follows from the Cauchy-Schwarz inequality that
\begin{equation*}
\langle \mathcal{R}_{\triangleright}\boldsymbol{u}^{k+1} - \mathcal{R}_{\triangleright}\boldsymbol{u}^{k}, \,\widetilde{\boldsymbol{u}}^k - \mathcal{J}_{\mathcal{C}^*\triangleright\mathcal{A}}\boldsymbol{u}^k \rangle
\leq {\textstyle\frac{2-\rho}{2\rho}}\big\|\mathcal{R}_{\triangleright}\boldsymbol{u}^{k+1}-\mathcal{R}_{\triangleright}\boldsymbol{u}^{k} \big\|^2 + {\textstyle\frac{\rho}{2(2-\rho)}}\big\|\widetilde{\boldsymbol{u}}^k - \mathcal{J}_{\mathcal{C}^*\triangleright\mathcal{A}}\boldsymbol{u}^k \big\|^2,
\end{equation*}
which, together with \eqref{ineq-R}, implies that
\begin{equation*}
\begin{aligned}
&\quad {\textstyle\frac{\rho}{2}}\big\|\mathcal{R}_{\triangleright}\boldsymbol{u}^{k+1}\big\|^2
+ {\textstyle\frac{\rho}{2}}\big\|\mathcal{R}_{\triangleright}\boldsymbol{u}^{k}\big\|^2
+ {\textstyle\frac{2-\rho}{2}}\big\|\mathcal{R}_{\triangleright}\boldsymbol{u}^{k+1}-\mathcal{R}_{\triangleright}\boldsymbol{u}^{k}\big\|^2 \\
&\leq {\textstyle\frac{\beta_k}{1-\beta_k}}\langle\mathcal{R}_{\triangleright}\boldsymbol{u}^{k+1}, \,\boldsymbol{u}^0-\boldsymbol{u}^{k+1} \rangle - \beta_k \langle \mathcal{R}_{\triangleright}\boldsymbol{u}^{k}, \,\boldsymbol{u}^0 - \boldsymbol{u}^k \rangle
+ \rho(1-\beta_k)\big\|\mathcal{R}_{\triangleright}\boldsymbol{u}^{k}\big\|^2 \\
&\quad + {\textstyle\frac{2-\rho}{2}}\big\|\mathcal{R}_{\triangleright}\boldsymbol{u}^{k+1}-\mathcal{R}_{\triangleright}\boldsymbol{u}^{k} \big\|^2
+ {\textstyle\frac{\rho^2}{2(2-\rho)}}\big\| \widetilde{\boldsymbol{u}}^k - \mathcal{J}_{\mathcal{C}^*\triangleright\mathcal{A}}\boldsymbol{u}^k \big\|^2
+ \rho\beta_k \langle \mathcal{R}_{\triangleright}\boldsymbol{u}^{k}, \,\widetilde{\boldsymbol{u}}^k - \mathcal{J}_{\mathcal{C}^*\triangleright\mathcal{A}}\boldsymbol{u}^k \rangle.
\end{aligned}
\end{equation*}
Thus, by rearranging the terms in this inequality, we obtain \eqref{ineq-any-beta}.

Furthermore, by substituting $\beta_k=\frac{1}{ k+r}$ into \eqref{ineq-any-beta}, we have that
\begin{equation*}
\begin{aligned}
&\quad {\textstyle\frac{\rho}{2}}\big\|\mathcal{R}_{\triangleright}\boldsymbol{u}^{k+1}\big\|^2
- {\textstyle\frac{1}{ k+r-1}}\langle\mathcal{R}_{\triangleright}\boldsymbol{u}^{k+1}, \,\boldsymbol{u}^0-\boldsymbol{u}^{k+1} \rangle
\\
&\leq {\textstyle\frac{\rho ( k+r-2)}{{2}( k+r)}}\big\|\mathcal{R}_{\triangleright}\boldsymbol{u}^{k}\big\|^2
- {\textstyle\frac{1}{ k+r}}\langle \mathcal{R}_{\triangleright}\boldsymbol{u}^{k}, \,\boldsymbol{u}^0-\boldsymbol{u}^k \rangle
\\
&\quad
+ {\textstyle\frac{\rho}{ k+r}}\langle\mathcal{R}_{\triangleright}\boldsymbol{u}^{k}, \,\widetilde{\boldsymbol{u}}^k-\mathcal{J}_{\mathcal{C}^*\triangleright\mathcal{A}}\boldsymbol{u}^k \rangle
+ {\textstyle\frac{\rho^2}{2(2-\rho)}}\big\|\widetilde{\boldsymbol{u}}^k-\mathcal{J}_{\mathcal{C}^*\triangleright\mathcal{A}}\boldsymbol{u}^k \big\|^2.
\end{aligned}
\end{equation*}
Multiplying both sides by $( k+r-1)( k+r)$, recalling the definition of $\mathcal{L}_k$, using $k+r\leq2(k+r-1)$ for $k\geq0$ and $r\geq2$, and applying the Cauchy-Schwarz inequality, we obtain that
\begin{equation*}
\begin{aligned}
\mathcal{L}_{k+1}
&\leq \mathcal{L}_k + \rho( k+r-1) \langle \mathcal{R}_{\triangleright}\boldsymbol{u}^{k}, \,\widetilde{\boldsymbol{u}}^k - \mathcal{J}_{\mathcal{C}^*\triangleright\mathcal{A}}\boldsymbol{u}^k \rangle
+ {\textstyle\frac{\rho^2( k+r-1)( k+r)}{2(2-\rho)}} \big\| \widetilde{\boldsymbol{u}}^k - \mathcal{J}_{\mathcal{C}^*\triangleright\mathcal{A}}\boldsymbol{u}^k \big\|^2 \\
&\leq \mathcal{L}_k + \rho( k+r-1) \langle \mathcal{R}_{\triangleright}\boldsymbol{u}^{k}, \,\widetilde{\boldsymbol{u}}^k - \mathcal{J}_{\mathcal{C}^*\triangleright\mathcal{A}}\boldsymbol{u}^k \rangle + {\textstyle\frac{\rho^2( k+r-1)^2}{2-\rho}} \big\| \widetilde{\boldsymbol{u}}^k - \mathcal{J}_{\mathcal{C}^*\triangleright\mathcal{A}}\boldsymbol{u}^k \big\|^2 \\
&\leq \mathcal{L}_k + {\textstyle\frac{\rho( k+r-1)}{2}\left( \frac{1}{6( k+r-1)} \big\| \mathcal{R}_{\triangleright}\boldsymbol{u}^{k}\big\|^2 + 6( k+r-1) \big\| \widetilde{\boldsymbol{u}}^k - \mathcal{J}_{\mathcal{C}^*\triangleright\mathcal{A}}\boldsymbol{u}^k \big\|^2 \right)}
\\
&\quad+ {\textstyle\frac{\rho^2( k+r-1)^2}{2-\rho}}\big\| \widetilde{\boldsymbol{u}}^k - \mathcal{J}_{\mathcal{C}^*\triangleright\mathcal{A}}\boldsymbol{u}^k \big\|^2 \\
&= \mathcal{L}_k + {\textstyle\frac{\rho}{12}} \big\| \mathcal{R}_{\triangleright}\boldsymbol{u}^{k}\big\|^2
+ {\textstyle\frac{2\rho(3-\rho)( k+r-1)^2}{2-\rho}} \big\| \widetilde{\boldsymbol{u}}^k - \mathcal{J}_{\mathcal{C}^*\triangleright\mathcal{A}}\boldsymbol{u}^k \big\|^2,
\end{aligned}
\end{equation*}
which completes the proof of statement (i).

{\textit{Case (ii)}. Set $\rho=2$ and $\widetilde{\boldsymbol{u}}^k=\mathcal{J}_{\mathcal{C}^*\triangleright\mathcal{A}}\boldsymbol{u}^k$ in \eqref{ineq-R}. Substituting $\beta_k=1/(k+r)$ and multiplying by $(k+r-1)(k+r)$ gives \eqref{ineq-suff-dec-rho2} directly from \eqref{def-Lk}.}
\hfill
\end{proof}

Based on the above descent property of the auxiliary sequence $\{\mathcal{L}_k\}_{k=0}^{\infty}$, the convergence rate of the \iHPGPPA is established in the following theorem. Our analysis follows a similar line of reasoning to that in \cite{d2020halpern, ltz2024inexact} for inexact Halpern iterations, but is more refined in that it explicitly characterizes the relationship between the relaxation parameter $\rho$ and the inexactness tolerance sequence $\{\varepsilon_k\}_{k=0}^{\infty}$. Furthermore, our analysis generalizes the selection of the anchoring parameter sequence $\{\beta_k\}_{k=0}^{\infty}$ to enhance the flexibility of the algorithm.

\begin{theorem}\textup{(\textbf{Convergence rate})}\label{thm-converge-rate}\label{thm-preconditioned-iAccGPPA}
Suppose that $\mathcal{M}$ is an admissible preconditioner and the solution set $\mathcal{S}$ is nonempty. Let $\mathcal{M}=\mathcal{C}\mathcal{C}^*$ be a decomposition of $\mathcal{M}$, $\{\boldsymbol{x}^{k}\}$ and $\{\boldsymbol{y}^{k}\}$ be generated by \eqref{Hal-PPP-inex} with fixed $\rho\in(0,2]$, fixed $r\geq2$, and $\beta_{k}=1/(k+r)$. Then, for any $k\geq1$, the following statements hold.
\begin{itemize}
        \item[{(i)}] {If $\rho\in(0,2)$, then, for any $\boldsymbol{x}^*\in\mathcal{S}$,}
        \begin{equation}\label{ineq-rate-r02}
             \big\| \boldsymbol{x}^k - {\mathcal{J}}_{\mathcal{A}}^{\mathcal{M}}\boldsymbol{x}^k \big\|_{\mathcal{M}}^2 \leq \frac{\Big[\frac{6(r-2)(r-1)+1}{2}+\frac{12}{\rho^2}\Big]\big\|\boldsymbol{x}^0 - \boldsymbol{x}^* \big\|_{\mathcal{M}}^2 + \frac{12(3-\rho)}{(2-\rho)}\sum_{i=0}^{k-1} (i+r-1)^2 \|\mathcal{C}^*\|_2^2\,\varepsilon_i^2}{(k+r-2)(k+r-1)}.
        \end{equation}
        {Consequently, if $\sum_{k=0}^{\infty}(k+r-1)^2\varepsilon_k^2<\infty$, then $\|\boldsymbol{x}^k-{\mathcal{J}}_{\mathcal{A}}^{\mathcal{M}}\boldsymbol{x}^k\|_{\mathcal{M}}^2=\mathcal{O}(1/k^2)$.}
        \item[{(ii)}] {If $\rho=2$, then, for any $\boldsymbol{x}^*\in\mathcal{S}$,}
        \begin{equation}\label{ineq-rate-r2}
        {\big\|\boldsymbol{x}^k-\mathcal{J}_{\mathcal{A}}^{\mathcal{M}}\boldsymbol{x}^k\big\|_{\mathcal{M}}
        \leq\frac{(r-1)\|\boldsymbol{x}^0-\boldsymbol{x}^*\|_{\mathcal{M}}+2\|\mathcal{C}^*\|_2\,\sum_{i=0}^{k-1}(i+r-1)\varepsilon_i}{k+r-1}.}
        \end{equation}
        {Consequently, if $\sum_{k=0}^{\infty}(k+r-1)\varepsilon_k<\infty$, then
        $\|\boldsymbol{x}^k-\mathcal{J}_{\mathcal{A}}^{\mathcal{M}}\boldsymbol{x}^k\|_{\mathcal{M}}^2=\mathcal{O}(1/k^2)$.}
\end{itemize}
\end{theorem}


\begin{proof}
{For any $\boldsymbol{x}^*\in\mathcal{S}$, Proposition \ref{parallel-comp} gives $\boldsymbol{u}^*=\mathcal{C}^*\boldsymbol{x}^*\in\operatorname{zero}(\mathcal{C}^*\triangleright\mathcal{A})$. By \eqref{auxi-seq},}
\begin{align*}
{\big\|\mathcal{R}_{\triangleright}\boldsymbol{u}^k\big\|}
&{=\big\|\mathcal{C}^*(\boldsymbol{x}^k-{\mathcal{J}}_{\mathcal{A}}^{\mathcal{M}}\boldsymbol{x}^k)\big\|
=\big\|\boldsymbol{x}^k-{\mathcal{J}}_{\mathcal{A}}^{\mathcal{M}}\boldsymbol{x}^k\big\|_{\mathcal{M}},}\\
{D:=\|\boldsymbol{u}^0-\boldsymbol{u}^*\|}
&{=\|\mathcal{C}^*(\boldsymbol{x}^0-\boldsymbol{x}^*)\|
=\|\boldsymbol{x}^0-\boldsymbol{x}^*\|_{\mathcal{M}}.}
\end{align*}

\textit{Statement (i)}.
From Lemma \ref{lemma-suff-dec}(i), we rewrite the inequality \eqref{ineq-suff-dec} as
\begin{equation*}
\mathcal{L}_{k+1} - \mathcal{L}_k \leq {\textstyle\frac{\rho}{12}} \big\| \mathcal{R}_{\triangleright}\boldsymbol{u}^{k} \big\|^2 + {\textstyle\frac{2\rho(3-\rho)}{(2-\rho)}} ( k+r-1)^2 \big\|\widetilde{\boldsymbol{u}}^k - \mathcal{J}_{\mathcal{C}^*\triangleright\mathcal{A}}\boldsymbol{u}^k \big\|^2.
\end{equation*}
By the inexactness criterion {\eqref{Hal-PPP-inex-u}}, summing over $0 \leq i \leq k-1$ ($k \geq 1$) yields that
\begin{equation}\label{ineq-diff-L}
\mathcal{L}_k = \mathcal{L}_0 + \sum_{i=0}^{k-1} (\mathcal{L}_{i+1} - \mathcal{L}_i)
\leq \mathcal{L}_0 +\frac{\rho}{12} \sum_{i=0}^{k-1} \big\| \mathcal{R}_{\triangleright}\boldsymbol{u}^i \big\|^2 + \frac{2\rho(3-\rho)}{(2-\rho)} \sum_{i=0}^{k-1} (i+r-1)^2 \|\mathcal{C}^*\|_2^2\,\varepsilon_i^2.
\end{equation}
Substituting the definition of $\mathcal{L}_k$ from \eqref{def-Lk} into \eqref{ineq-diff-L}, we obtain
\begin{equation}\label{ineq-thm-rate-RBound}
\begin{aligned}
    &\frac{\rho ( k+r-2)( k+r-1)}{2} \big\| \mathcal{R}_{\triangleright}\boldsymbol{u}^{k} \big\|^2 \\
    \leq\, &\mathcal{L}_0 + \frac{\rho}{12} \sum_{i=0}^{k-1} \big\| \mathcal{R}_{\triangleright}\boldsymbol{u}^i \big\|^2 + \frac{2\rho(3-\rho)}{(2-\rho)} \sum_{i=0}^{k-1} (i+r-1)^2 \|\mathcal{C}^*\|_2^2\,\varepsilon_i^2 + ( k+r-1) \langle \mathcal{R}_{\triangleright}\boldsymbol{u}^{k}, \boldsymbol{u}^0 - \boldsymbol{u}^k \rangle.
\end{aligned}
\end{equation}
To bound the inner product term $\langle \mathcal{R}_{\triangleright}\boldsymbol{u}^k, \boldsymbol{u}^0 - \boldsymbol{u}^k \rangle$, we use the 1-co-coercivity of $\mathcal{R}_{\triangleright}$ and the fact that $\mathcal{R}_{\triangleright}\boldsymbol{u}^* = 0$. For any $\boldsymbol{u}^* \in \mathcal{C}^*\mathcal{S}$, we have
\begin{equation*}
\langle \mathcal{R}_{\triangleright}\boldsymbol{u}^k - \mathcal{R}_{\triangleright}\boldsymbol{u}^*, \boldsymbol{u}^k - \boldsymbol{u}^* \rangle \geq \|\mathcal{R}_{\triangleright}\boldsymbol{u}^k - \mathcal{R}_{\triangleright}\boldsymbol{u}^*\|^2 \geq 0,
\end{equation*}
which {implies} that
$
\langle \mathcal{R}_{\triangleright}\boldsymbol{u}^k, \boldsymbol{u}^* - \boldsymbol{u}^k \rangle \leq 0.
$
By the Cauchy-Schwarz inequality, we obtain
\begin{equation}\label{bound-R-x0}
\langle \mathcal{R}_{\triangleright}\boldsymbol{u}^{k}, \boldsymbol{u}^0 - \boldsymbol{u}^k \rangle = \langle \mathcal{R}_{\triangleright}\boldsymbol{u}^{k}, \boldsymbol{u}^0 - \boldsymbol{u}^* \rangle + \langle \mathcal{R}_{\triangleright}\boldsymbol{u}^{k}, \boldsymbol{u}^* - \boldsymbol{u}^k \rangle
\leq \langle \mathcal{R}_{\triangleright}\boldsymbol{u}^{k}, \boldsymbol{u}^0 - \boldsymbol{u}^* \rangle
\leq  D\big\| \mathcal{R}_{\triangleright}\boldsymbol{u}^{k} \big\|.
\end{equation}
Substituting \eqref{bound-R-x0} into \eqref{ineq-thm-rate-RBound} and dividing both sides by $\frac{\rho ( k+r-2)( k+r-1)}{2}$, we have
\begin{equation}\label{ineq-Rk}
\begin{aligned}
\big\| \mathcal{R}_{\triangleright}\boldsymbol{u}^{k} \big\|^2
\leq \,&\frac{2\mathcal{L}_0}{\rho( k+r-2)( k+r-1)} +\frac{1}{6( k+r-2)( k+r-1)} \sum_{i=0}^{k-1} \big\| \mathcal{R}_{\triangleright}\boldsymbol{u}^i \big\|^2 \\
&+ \frac{4(3-\rho)}{(2-\rho)( k+r-2)( k+r-1)} \sum_{i=0}^{k-1} (i+r-1)^2 \|\mathcal{C}^*\|_2^2\,\varepsilon_i^2 + \frac{2D}{\rho ( k+r-2)} \big\| \mathcal{R}_{\triangleright}\boldsymbol{u}^{k} \big\|.
\end{aligned}
\end{equation}
Applying Young's inequality to the last term gives
\begin{equation*}
\frac{2D}{\rho ( k+r-2)} {\big\| \mathcal{R}_{\triangleright}\boldsymbol{u}^{k} \big\|} \leq \frac{1}{2} \big\| \mathcal{R}_{\triangleright}\boldsymbol{u}^{k} \big\|^2 + \frac{1}{2} \left( \frac{2D}{\rho ( k+r-2)} \right)^2 = \frac{1}{2} \big\| \mathcal{R}_{\triangleright}\boldsymbol{u}^{k} \big\|^2 + \frac{2D^2}{\rho^2 ( k+r-2)^2}.
\end{equation*}
Substituting the above inequality into \eqref{ineq-Rk} and rearranging the terms, we obtain
\begin{equation}\label{ineq-ak}
\begin{aligned}
\big\| \mathcal{R}_{\triangleright}\boldsymbol{u}^{k} \big\|^2
\leq\,
&\frac{1}{( k+r-2)( k+r-1)}
\left(
\begin{array}{l}\frac{4\mathcal{L}_0}{\rho}
+\frac{1}{3} \sum_{i=0}^{k-1} \big\| \mathcal{R}_{\triangleright}\boldsymbol{u}^{i} \big\|^2
\\[3pt]
+ \frac{8(3-\rho)}{(2-\rho)} \sum_{i=0}^{k-1} (i+r-1)^2 \|\mathcal{C}^*\|_2^2\,\varepsilon_i^2
+ \frac{4D^2(k+r-1)}{\rho^2 ( k+r-2)}
\end{array}\right).
\end{aligned}
\end{equation}
For notational simplicity, we define
\begin{equation}\label{denote-b-C}
\begin{aligned}
b_k := {\textstyle\sum_{i=0}^{k-1}} (i+r-1)^2 \|\mathcal{C}^*\|_2^2\,\varepsilon_i^2,\quad \text{and} \quad
C_k :={\textstyle\frac{6}{\rho}}\mathcal{L}_0  + {\textstyle\frac{12(3-\rho)}{(2-\rho)}} b_k + {\textstyle\frac{\rho^2+24}{2\rho^2}} D^2.
\end{aligned}
\end{equation}
The sequence $\{C_k\}_{k=1}^{\infty}$ is non-decreasing since $\{b_k\}_{k=1}^{\infty}$ is non-decreasing and $\rho \in (0, 2)$.

We now prove by mathematical induction that for all $k \geq 1$,
\begin{equation}\label{ineq-rate-r02-k0}
\|\mathcal{R}_{\triangleright}\boldsymbol{u}^k\|^2 \leq \frac{C_k}{( k+r-2)( k+r-1)}.
\end{equation}
For $k=1$, from \eqref{ineq-ak}, we have that
\begin{equation*}%
\begin{aligned}
\big\| \mathcal{R}_{\triangleright}\boldsymbol{u}^{1} \big\|^2
&\leq {\frac{1}{(r-1)r}} \left[\frac{4}{\rho}\mathcal{L}_0 + \frac{1}{3} \big\| \mathcal{R}_{\triangleright}\boldsymbol{u}^{0} \big\|^2 + \frac{8(3-\rho)}{(2-\rho)} b_1\right] + \frac{4D^2}{\rho^2(r-1)^2}\\
&\leq {\frac{1}{(r-1)r}} \left[\frac{4}{\rho}\mathcal{L}_0 + \frac{1}{3} \big\| \mathcal{R}_{\triangleright}\boldsymbol{u}^{0} \big\|^2 + \frac{8(3-\rho)}{(2-\rho)} b_1 + \frac{8D^2}{\rho^2}\right]\\
&\leq {\frac{1}{(r-1)r}} \left[\frac{4}{\rho}\mathcal{L}_0 + \frac{8(3-\rho)}{(2-\rho)} b_1 + \left(\frac{1}{3}+ \frac{8}{\rho^2}\right)D^2\right]\;
\leq\; \frac{C_1}{(r-1)r},
\end{aligned}
\end{equation*}
where the second inequality follows from $\frac{r}{r-1}\leq2$ for all $r\geq2$, the third inequality follows from $\|\mathcal{R}_{\triangleright}\boldsymbol{u}^0\|^2  =\|\mathcal{R}_{\triangleright}\boldsymbol{u}^0-\mathcal{R}_{\triangleright}\boldsymbol{u}^*\|^2 \leq \|\boldsymbol{u}^0 - \boldsymbol{u}^*\|^2 = D^2$ and the last inequality follows from the definition of $C_1$ in \eqref{denote-b-C}.
Assume \eqref{ineq-rate-r02-k0} holds for $j=1,\dots,k-1$. Substituting the induction hypothesis into \eqref{ineq-ak}, we derive that
\begin{align*}
\big\| \mathcal{R}_{\triangleright}\boldsymbol{u}^{k} \big\|^2
&\leq \frac{\frac{4}{\rho}\mathcal{L}_0 + \sum_{j=1}^{k-1} \frac{C_j}{3(j+r-2)(j+r-1)} +\frac{D^2}{3} + \frac{8(3-\rho)}{(2-\rho)} b_k + \frac{8D^2}{\rho^2}}{( k+r-2)( k+r-1)}.
\end{align*}
Since $C_j \leq C_k$ and $\sum_{j=1}^{k-1}\frac{1}{(j+r-2)(j+r-1)} = \frac{1}{r-1} - \frac{1}{k+r-2} = \frac{k-1}{(r-1)(k+r-2)} \leq 1$ for $k \geq 1$ and $r \geq 2$, we simplify
\begin{align*}
\big\| \mathcal{R}_{\triangleright}\boldsymbol{u}^{k} \big\|^2
&\leq \frac{\frac{4}{\rho}\mathcal{L}_0 + \frac{C_k}{3} +\frac{D^2}{3} + \frac{8(3-\rho)}{(2-\rho)} b_k + \frac{8D^2}{\rho^2}}{( k+r-2)( k+r-1)} \\
&= \frac{1}{( k+r-2)( k+r-1)} \left( \frac{C_k}{3} +\frac{2C_k}{3}\right)
\;=\; \frac{C_k}{( k+r-2)( k+r-1)},
\end{align*}
where the first equality follows directly from the definition of $C_k$ in \eqref{denote-b-C}.
This completes the induction proof of \eqref{ineq-rate-r02-k0}.

Recalling that $\mathcal{L}_0 = \frac{\rho (r-2)(r-1)}{2}\big\| \mathcal{R}_{\triangleright}\boldsymbol{u}^{0} \big\|^2\leq\frac{\rho (r-2)(r-1)}{2}D^2$, substituting $\mathcal{L}_0$ into the definition of $C_k$ and {using the inequality \eqref{ineq-rate-r02-k0} yields} the required
inequality \eqref{ineq-rate-r02}. Moreover, if $\sum_{k=0}^{\infty}(k+r-1)^2\varepsilon_k^2 < \infty$, then $b_k$ is bounded, so $C_k$ is bounded. This implies $\|\mathcal{R}_{\triangleright}\boldsymbol{u}^k\|^2 = \mathcal{O}(1/k^2)$, which completes the proof of Statement (i).

\textit{Statement (ii)}.
Consider the exact sequence
$\boldsymbol{v}^0=\boldsymbol{u}^0$,
$\boldsymbol{v}^{k+1}=\beta_k\boldsymbol{u}^0+(1-\beta_k)\mathcal{F}_2\boldsymbol{v}^k$.
Let $\mathcal{L}_k^{\rm ex}$ denote \eqref{def-Lk} along this sequence with $\rho=2$.
Lemma \ref{lemma-suff-dec}(ii) gives
$\mathcal{L}_k^{\rm ex}\leq\mathcal{L}_0^{\rm ex}=(r-2)(r-1)\|\mathcal{R}_{\triangleright}\boldsymbol{u}^0\|^2$.
By 1-co-coercivity,
\[
 \langle\mathcal{R}_{\triangleright}\boldsymbol{v}^k,\boldsymbol{u}^0-\boldsymbol{v}^k\rangle
 \leq D\|\mathcal{R}_{\triangleright}\boldsymbol{v}^k\|
       -\|\mathcal{R}_{\triangleright}\boldsymbol{v}^k\|^2.
\]
Hence, setting $t_k=(k+r-1)\|\mathcal{R}_{\triangleright}\boldsymbol{v}^k\|$, we obtain
\[
 t_k^2-Dt_k\leq\mathcal{L}_k^{\rm ex}
 \leq\mathcal{L}_0^{\rm ex}\leq(r-2)(r-1)D^2,
 \qquad
 \|\mathcal{R}_{\triangleright}\boldsymbol{v}^k\|\leq\frac{(r-1)D}{k+r-1}.
\]
Nonexpansivity of $\mathcal{F}_2$ and the inexactness criterion give
\[
 \|\boldsymbol{u}^{k+1}-\boldsymbol{v}^{k+1}\|
 \leq(1-\beta_k)\big(\|\boldsymbol{u}^k-\boldsymbol{v}^k\|
       +2\|\mathcal{C}^*\|_2\,\varepsilon_k\big).
\]
Since $\prod_{j=i}^{k-1}(1-\beta_j)=(i+r-1)/(k+r-1)$, iteration yields
\[
 \|\boldsymbol{u}^k-\boldsymbol{v}^k\|
 \leq\frac{2\|\mathcal{C}^*\|_2}{k+r-1}
       \sum_{i=0}^{k-1}(i+r-1)\varepsilon_i.
\]
Combining these two bounds with the 1-Lipschitz property of $\mathcal{R}_{\triangleright}$ proves \eqref{ineq-rate-r2}. Its numerator is bounded under the stated weighted summability condition, which proves the squared residual rate.
\hfill
\end{proof}

The following example shows that, at the reflected endpoint $\rho=2$, the weighted summability condition in Theorem \ref{thm-preconditioned-iAccGPPA}(ii) is sharp for guaranteeing an $\mathcal{O}(1/k^{2})$ convergence rate for the squared fixed-point residual, uniformly over all admissible error directions.

\begin{example}[Sharp endpoint inexactness condition]\label{ex-endpoint-errors}
Let $\mathcal{M}=\mathcal{I}$ on $\mathbb{R}$ and let $\mathcal{A}=N_{\{0\}}$ be the normal cone to $\{0\}$. Then $\mathcal{J}_{\mathcal{A}}^{\mathcal{M}}\equiv0$, and the unique solution is $0$. Fix $r\geq2$ and take $\rho=2$, $\boldsymbol{x}^{0}=0$, and $\beta_{k}=1/(k+r)$. For any prescribed nonnegative tolerance sequence $\{\varepsilon_{k}\}$, choose $\boldsymbol{y}^{k}=(-1)^{k}\varepsilon_{k}$. These points satisfy the inexactness criterion, and the resulting iterates satisfy
\begin{equation*}
\big\|\boldsymbol{x}^{k}-\mathcal{J}_{\mathcal{A}}^{\mathcal{M}}\boldsymbol{x}^{k}\big\|
=\lvert\boldsymbol{x}^{k}\rvert
=\frac{2}{k+r-1}\sum_{i=0}^{k-1}(i+r-1)\varepsilon_{i},
\qquad k\geq1.
\end{equation*}
Thus, for this example, an $\mathcal{O}(1/k^{2})$ convergence rate for the squared fixed-point residual holds if and only if $\sum_{k=0}^{\infty}(k+r-1)\varepsilon_{k}<\infty$. Together with Theorem \ref{thm-preconditioned-iAccGPPA}(ii), this establishes the necessity and sufficiency of the weighted summability condition for guaranteeing this rate, uniformly over all admissible error directions. Indeed, choosing $\varepsilon_{k}=(k+1)^{-2}$ in this example yields {$\lvert\boldsymbol{x}^{k}\rvert=\Theta(\log(k)/k)$}. The iterates therefore still converge to the solution $0$. However, the $\mathcal{O}(1/k^{2})$ convergence rate for the squared fixed-point residual is not preserved.
\end{example}

The squared fixed-point residual bound in Theorem \ref{thm-preconditioned-iAccGPPA}(i) contains the factor $12(3-\rho)/(2-\rho)$ in its error term. This factor becomes unbounded as $\rho\uparrow2$. To examine whether this growth is unavoidable, we construct inexact resolvent points for the scalar problem in Example \ref{ex-endpoint-errors} and compute the resulting squared fixed-point residual exactly. We take $\boldsymbol{x}^{0}=0$, so that the exact iterates remain at the solution $0$. Any nonzero squared fixed-point residual is therefore due solely to the resolvent errors.

\begin{proposition}[Error amplification factor]\label{prop-error-amplification}
Let $\mathcal{M}=\mathcal{I}$ and $\mathcal{A}=N_{\{0\}}$ on $\mathbb{R}$, i.e., $\mathcal{A}(x) = \emptyset$ if $x\not=0$ and $\mathcal{A}(x) = \mathbb{R}$ if $x=0$. Note that $\mathcal{A}$ is maximally monotone since it is the normal cone to a nonempty closed convex set. We have that $\mathcal{J}_{\mathcal{A}}^{\mathcal{M}}\equiv0$. Fix a real $r\geq2$ and $0<\rho\leq2$. For every integer $k\geq1$, there exist inexact resolvent points $\boldsymbol{y}^{0},\ldots,\boldsymbol{y}^{k-1}\in\mathbb{R}$, not all zero, such that the iterate $\boldsymbol{x}^{k}$ generated by \eqref{Hal-PPP-inex} from $\boldsymbol{x}^{0}=0$, with $\beta_{i}=1/(i+r)$ and $\varepsilon_{i}=\lvert\boldsymbol{y}^{i}\rvert$, satisfies
\begin{equation}\label{eq-error-amplification}
\lvert\boldsymbol{x}^{k}\rvert^{2}
=\frac{\sum_{i=0}^{k-1}(i+r-1)^{2}\varepsilon_{i}^{2}}{(k+r-1)^{2}}
\begin{cases}
\displaystyle\frac{\rho}{2-\rho}\left[1-(1-\rho)^{2k}\right],&0<\rho<2,\\[5pt]
4k,&\rho=2.
\end{cases}
\end{equation}
\end{proposition}
\begin{proof}
Fix $k\geq1$. Since $\boldsymbol{x}^{0}=0$, multiplying the update at iteration $i$ by $i+r$ gives
$(i+r)\boldsymbol{x}^{i+1}=(1-\rho)(i+r-1)\boldsymbol{x}^{i}+\rho(i+r-1)\boldsymbol{y}^{i}$.

First, suppose that $\rho=1$. Choose $\boldsymbol{y}^{i}=0$ for $0\leq i<k-1$ and $\boldsymbol{y}^{k-1}=1/(k+r-2)$. The final update gives $\boldsymbol{x}^{k}=1/(k+r-1)$, while $\sum_{i=0}^{k-1}(i+r-1)^{2}\varepsilon_{i}^{2}=1$. These identities prove \eqref{eq-error-amplification} for $\rho=1$.

Now suppose that $\rho\neq1$ and choose
\begin{equation*}
\boldsymbol{y}^{i}=\frac{(1-\rho)^{k-1-i}}{i+r-1},
\qquad 0\leq i<k.
\end{equation*}
These points are not all zero, since $\boldsymbol{y}^{k-1}=1/(k+r-2)>0$.
Iterating this recurrence from $\boldsymbol{x}^{0}=0$ yields
\begin{equation*}
(k+r-1)\boldsymbol{x}^{k}
=\rho{\textstyle\sum_{i=0}^{k-1}}(1-\rho)^{k-1-i}(i+r-1)\boldsymbol{y}^{i}.
\end{equation*}
Substituting the chosen points and using $\varepsilon_{i}=\lvert\boldsymbol{y}^{i}\rvert$ gives
\begin{equation*}
\begin{aligned}
(k+r-1)\boldsymbol{x}^{k}
&=\rho{\textstyle\sum_{j=0}^{k-1}}(1-\rho)^{2j},\\
{\textstyle\sum_{i=0}^{k-1}}(i+r-1)^{2}\varepsilon_{i}^{2}
&={\textstyle\sum_{i=0}^{k-1}}(1-\rho)^{2(k-1-i)}
={\textstyle\sum_{j=0}^{k-1}}(1-\rho)^{2j}.
\end{aligned}
\end{equation*}
Consequently,
\begin{equation*}
\lvert\boldsymbol{x}^{k}\rvert^{2}
=\frac{\sum_{i=0}^{k-1}(i+r-1)^{2}\varepsilon_{i}^{2}}{(k+r-1)^{2}}\left(\rho^{2}\sum_{j=0}^{k-1}(1-\rho)^{2j}\right).
\end{equation*}
For $0<\rho<2$, the factor in parentheses equals $\rho[1-(1-\rho)^{2k}]/(2-\rho)$ by the geometric-sum formula. For $\rho=2$, it equals $4k$. This proves \eqref{eq-error-amplification}.
\end{proof}

For fixed $0<\rho<2$, the factor $\rho[1-(1-\rho)^{2k}]/(2-\rho)$ in \eqref{eq-error-amplification} is bounded in $k$ and tends to $\rho/(2-\rho)$ as $k\to\infty$. Applying a bound of the form \eqref{ineq-rate-r02} to the constructed errors and letting $k\to\infty$ shows that its error coefficient must be at least $\rho/(2-\rho)$. Hence the $(2-\rho)^{-1}$ growth is unavoidable for bounds uniform over iterations and error choices.

For each fixed $k$, the same factor approaches $4k$ as $\rho\uparrow2$, linking the deterioration in $\rho$ to linear growth in $k$ at the endpoint. Therefore, the geometric weights have unit magnitude, and alternating error signs make weighted errors accumulate without cancellation. As Example \ref{ex-endpoint-errors} shows, retaining the $\mathcal{O}(1/k^{2})$ convergence rate for the squared fixed-point residual uniformly over all admissible error directions then requires $\sum_{i=0}^{\infty}(i+r-1)\varepsilon_{i}<\infty$, stronger than the weighted square-summability condition that is sufficient for the case $\rho\in (0,2)$.

\begin{corollary}[Squared Euclidean fixed-point residual rates]\label{cor-euclidean-residual}
Suppose that the assumptions of Theorem \ref{thm-preconditioned-iAccGPPA} hold and that $(\mathcal{M}+\mathcal{A})^{-1}$ is Lipschitz continuous. Assume further that $\sum_{k=0}^{\infty}(k+r-1)^{2}\varepsilon_{k}^{2}<\infty$ when $0<\rho<2$, and that $\sum_{k=0}^{\infty}(k+r-1)\varepsilon_{k}<\infty$ when $\rho=2$. Then, the following statements hold.
\begin{itemize}
\item[\textup{(i)}] For every $0<\rho\leq2$, the inexact resolvent points satisfy
\begin{equation*}
\|\boldsymbol{y}^{k}-\mathcal{J}_{\mathcal{A}}^{\mathcal{M}}\boldsymbol{y}^{k}\|^{2}=\mathcal{O}(1/k^{2}).
\end{equation*}
\item[\textup{(ii)}] If $0<\rho<2$, the iterates also satisfy
\begin{equation*}
\|\boldsymbol{x}^{k}-\mathcal{J}_{\mathcal{A}}^{\mathcal{M}}\boldsymbol{x}^{k}\|^{2}=\mathcal{O}(1/k^{2}).
\end{equation*}
\end{itemize}
\end{corollary}
\begin{proof}
\textit{Statement (i).} Let $L$ be a Lipschitz constant of $(\mathcal{M}+\mathcal{A})^{-1}$. Since $\mathcal{M}=\mathcal{C}\mathcal{C}^{*}$ and $\|\mathcal{C}^{*}\boldsymbol{v}\|=\|\boldsymbol{v}\|_{\mathcal{M}}$, the inexactness criterion and the triangle inequality give
\begin{equation}\label{eq-computed-point-euclidean}
\begin{aligned}
\|\boldsymbol{y}^{k}-\mathcal{J}_{\mathcal{A}}^{\mathcal{M}}\boldsymbol{y}^{k}\|
&\leq\varepsilon_{k}+L\|\mathcal{C}\|_2\|\boldsymbol{x}^{k}-\boldsymbol{y}^{k}\|_{\mathcal{M}}\\
&\leq L\|\mathcal{C}\|_2\|\boldsymbol{x}^{k}-\mathcal{J}_{\mathcal{A}}^{\mathcal{M}}\boldsymbol{x}^{k}\|_{\mathcal{M}}
+(1+L\|\mathcal{C}\|_2\|\mathcal{C}^{*}\|_2)\varepsilon_{k}.
\end{aligned}
\end{equation}
Each of the two weighted summability conditions implies $(k+r-1)\varepsilon_{k}\to0$, and hence $\varepsilon_{k}=o(1/k)$. By Theorem \ref{thm-preconditioned-iAccGPPA}, the seminorm residual on the right-hand side is $\mathcal{O}(1/k)$ for every $0<\rho\leq2$. Squaring \eqref{eq-computed-point-euclidean} proves statement (i).

\textit{Statement (ii).} Suppose that $0<\rho<2$.
The weighted summability condition and the Cauchy--Schwarz inequality imply $\sum_k\varepsilon_k<\infty$ and $\varepsilon_k=o(1/k)$. Proposition \ref{propo-bound} and Theorem \ref{thm-preconditioned-iAccGPPA}(i) therefore give boundedness of $\{\boldsymbol{u}^k\}$ and $\|\mathcal{R}_{\triangleright}\boldsymbol{u}^k\|=\mathcal{O}(1/k)$. Combining Lemma \ref{lemma-z}(i) with \eqref{Hal-PPP-inex-u} yields $\|\boldsymbol{u}^{k+1}-\boldsymbol{u}^k\|=\mathcal{O}(1/k)$. By the Lipschitz assumption and
\begin{equation*}
 {\mathcal{J}}_{\mathcal{A}}^{\mathcal{M}}\boldsymbol{x}^k
 =(\mathcal{M}+\mathcal{A})^{-1}\mathcal{C}\boldsymbol{u}^k,
\end{equation*}
the sequence $\{{\mathcal{J}}_{\mathcal{A}}^{\mathcal{M}}\boldsymbol{x}^k\}$ is bounded and
$\|{\mathcal{J}}_{\mathcal{A}}^{\mathcal{M}}\boldsymbol{x}^{k+1}-{\mathcal{J}}_{\mathcal{A}}^{\mathcal{M}}\boldsymbol{x}^k\|=\mathcal{O}(1/k)$.
Using \eqref{Hal-PPP-inex}, we obtain
\begin{align*}
 &\big\|\boldsymbol{x}^{k+1}-{\mathcal{J}}_{\mathcal{A}}^{\mathcal{M}}\boldsymbol{x}^{k+1}\big\|
 \;\leq\; |1-\rho|\big\|\boldsymbol{x}^{k}-{\mathcal{J}}_{\mathcal{A}}^{\mathcal{M}}\boldsymbol{x}^{k}\big\|
 +\big\|{\mathcal{J}}_{\mathcal{A}}^{\mathcal{M}}\boldsymbol{x}^{k+1}-{\mathcal{J}}_{\mathcal{A}}^{\mathcal{M}}\boldsymbol{x}^{k}\big\|
 +\beta_k\big\|\boldsymbol{x}^0-{\mathcal{J}}_{\mathcal{A}}^{\mathcal{M}}\boldsymbol{x}^{k}\big\|+\rho\varepsilon_k
 \\
 &
 \leq |1-\rho|\big\|\boldsymbol{x}^{k}-{\mathcal{J}}_{\mathcal{A}}^{\mathcal{M}}\boldsymbol{x}^{k}\big\|+\frac{B}{k+1}
\end{align*}
for some constant $B>0$ and every $k\geq0$. If $\rho=1$, statement (ii) follows immediately. Otherwise, iterating this inequality gives
\begin{align*}
\big\|\boldsymbol{x}^{k}-{\mathcal{J}}_{\mathcal{A}}^{\mathcal{M}}\boldsymbol{x}^{k}\big\|
&\leq |1-\rho|^k\big\|\boldsymbol{x}^{0}-{\mathcal{J}}_{\mathcal{A}}^{\mathcal{M}}\boldsymbol{x}^{0}\big\|
+B\sum_{i=0}^{k-1}\frac{|1-\rho|^{k-1-i}}{i+1}=\mathcal{O}(1/k),
\end{align*}
where the last bound follows by splitting the sum at $\lfloor k/2\rfloor$ and using $|1-\rho|<1$. {Squaring this bound proves statement (ii).}
\end{proof}

Example \ref{rem-singular-endpoint} shows that, when $\mathcal{M}$ is singular, the squared Euclidean fixed-point residual bound for $\boldsymbol{x}^{k}$ in Corollary \ref{cor-euclidean-residual}(ii) cannot in general be extended to $\rho=2$, even with exact resolvent evaluations.

We next give a computable bound on the squared fixed-point residuals in the $\mathcal{M}$-seminorm at both $\boldsymbol{x}^{k}$ and $\boldsymbol{y}^{k}$. This bound provides a stopping criterion under which either point may be returned.

\begin{corollary}[Computable seminorm residual bound]\label{cor-output-certificate}
Suppose that the assumptions of Theorem \ref{thm-preconditioned-iAccGPPA} hold. Assume that $\sum_{k=0}^{\infty}(k+r-1)^{2}\varepsilon_{k}^{2}<\infty$ when $0<\rho<2$, and that $\sum_{k=0}^{\infty}(k+r-1)\varepsilon_{k}<\infty$ when $\rho=2$. Then
\begin{equation}\label{eq-computable-squared-residual}
\begin{aligned}
\max\left\{
\|\boldsymbol{x}^{k}-\mathcal{J}_{\mathcal{A}}^{\mathcal{M}}\boldsymbol{x}^{k}\|_{\mathcal{M}}^{2},\,
\|\boldsymbol{y}^{k}-\mathcal{J}_{\mathcal{A}}^{\mathcal{M}}\boldsymbol{y}^{k}\|_{\mathcal{M}}^{2}
\right\}
\leq\left(\|\boldsymbol{x}^{k}-\boldsymbol{y}^{k}\|_{\mathcal{M}}+\sqrt{\|\mathcal{M}\|_2}\varepsilon_{k}\right)^{2}
=\mathcal{O}(1/k^{2}).
\end{aligned}
\end{equation}
\end{corollary}
\begin{proof}
The inexactness criterion and the reverse triangle inequality give
\begin{equation*}
\left\lvert\|\boldsymbol{x}^{k}-\boldsymbol{y}^{k}\|_{\mathcal{M}}
-\|\boldsymbol{x}^{k}-\mathcal{J}_{\mathcal{A}}^{\mathcal{M}}\boldsymbol{x}^{k}\|_{\mathcal{M}}\right\rvert
\leq\sqrt{\|\mathcal{M}\|_2}\,\varepsilon_{k}.
\end{equation*}
By Proposition \ref{parallel-comp}, $\mathcal{J}_{\mathcal{A}}^{\mathcal{M}}$ is nonexpansive in the $\mathcal{M}$-seminorm. Hence
\begin{equation*}
\begin{aligned}
\|\boldsymbol{y}^{k}-\mathcal{J}_{\mathcal{A}}^{\mathcal{M}}\boldsymbol{y}^{k}\|_{\mathcal{M}}
&\leq\|\boldsymbol{y}^{k}-\mathcal{J}_{\mathcal{A}}^{\mathcal{M}}\boldsymbol{x}^{k}\|_{\mathcal{M}}
+\|\mathcal{J}_{\mathcal{A}}^{\mathcal{M}}\boldsymbol{x}^{k}-\mathcal{J}_{\mathcal{A}}^{\mathcal{M}}\boldsymbol{y}^{k}\|_{\mathcal{M}}
\;\leq\;\sqrt{\|\mathcal{M}\|_2}\,\varepsilon_{k}+\|\boldsymbol{x}^{k}-\boldsymbol{y}^{k}\|_{\mathcal{M}}.
\end{aligned}
\end{equation*}
Both weighted summability conditions imply $\varepsilon_{k}=o(1/k)$. Theorem \ref{thm-preconditioned-iAccGPPA} and the first inequality therefore give $\|\boldsymbol{x}^{k}-\boldsymbol{y}^{k}\|_{\mathcal{M}}=\mathcal{O}(1/k)$. Squaring the preceding bounds proves \eqref{eq-computable-squared-residual}.
\end{proof}

The inequalities in the proof hold at every iteration without the weighted summability conditions. Thus, for a target accuracy $\epsilon>0$, the test $\|\boldsymbol{x}^{k}-\boldsymbol{y}^{k}\|_{\mathcal{M}}+\sqrt{\|\mathcal{M}\|_2}\,\varepsilon_{k}\leq\epsilon$ allows either $\boldsymbol{x}^{k}$ or $\boldsymbol{y}^{k}$ to be returned with a squared fixed-point residual in the $\mathcal{M}$-seminorm of at most $\epsilon^{2}$.  When $\mathcal{M}\succ0$, norm equivalence also gives $\|\boldsymbol{x}^{k}-\mathcal{J}_{\mathcal{A}}^{\mathcal{M}}\boldsymbol{x}^{k}\|^{2}=\mathcal{O}(1/k^{2})$ for all $0<\rho\leq2$ under the same weighted summability conditions.

Before closing this section, we highlight several consequences of the convergence-rate bounds. First, \iHPGPPA recovers the inexact Halpern proximal framework of \cite{ltz2024inexact} as the special case $\mathcal{M}=\mathcal{I}$, $\rho = 1$ and $r = 2$, while extending the theory to the whole relaxation interval $\rho\in(0,2]$. Second, in the exact setting, the bound in \eqref{ineq-rate-r02} shows that increasing $\rho$ improves the worst-case constant throughout the interior regime $\rho\in(0,2)$, which is consistent with the practical appeal of over-relaxation. Third, the endpoint $\rho=2$ is not merely a limiting case of the interior analysis; it instead obeys a different bound, and our analysis uses a strictly stronger inexactness condition to preserve the accelerated rate. This shows that the reflected resolvent case is theoretically more delicate in the presence of inexact subproblem solves. Finally, the tunable anchoring rule $\beta_k = \frac{1}{k+r}$ with $r \geq 2$ provides an additional degree of flexibility. Choosing $r>2$ weakens the anchoring effect relative to the classical choice $\frac{1}{k+2}$, which is possibly beneficial in practice while preserving the $\mathcal{O}(1/k^2)$ worst-case rate.

\section{Applications}\label{sec-App}

In this section, we apply our \iHPGPPA to pADMM and preconditioned PDHG to derive their inexact Halpern-accelerated versions.

\subsection{Inexact accelerated pADMM}\label{sec-App-pADMM}

Consider the following convex linearly constrained minimization problem:
\begin{equation}\label{eq-padmm-P}
\min\big\{f(y)+g(z)\mid Ay+Bz=c\big\}.
\end{equation}
Here $f$ and $g$ are closed proper convex functions defined on finite-dimensional Euclidean spaces. The linear operators $A$ and $B$ have adjoints $A^*$ and $B^*$. Let $w:=(y,z,x)${, where $x$ is the multiplier for the constraint $Ay+Bz=c$}. The KKT conditions for \eqref{eq-padmm-P} can be written as
\begin{equation}\label{eq-padmm-KKT}
0\in\mathcal{T}(w):=
\begin{pmatrix}
\partial f(y)+A^*x\\
\partial g(z)+B^*x\\
c-Ay-Bz
\end{pmatrix}.
\end{equation}
We assume that \eqref{eq-padmm-KKT} has a solution. The operator $\mathcal{T}$ is maximally monotone. For $\sigma>0$, the augmented Lagrangian function associated with \eqref{eq-padmm-P} is given by
\begin{equation*}
L_\sigma(y,z;x)=f(y)+g(z)
+\frac{\sigma}{2}\|Ay+Bz-c+\sigma^{-1}x\|^2
-\frac{1}{2\sigma}\|x\|^2.
\end{equation*}
{Let $T_y$ and $T_z$ be self-adjoint positive semidefinite operators such that}
\begin{equation*}
M_y:=T_y+A^*A\succ0,\qquad M_z:=T_z+B^*B\succ0.
\end{equation*}
The pADMM template \cite{xcl2018generalized}
for solving \eqref{eq-padmm-P} is given as follows:
\begin{equation}\label{eq-padmm-template}
\begin{aligned}
\bar{z}^{k+1}
&=\mathop{\rm argmin}_{{z}}\left\{L_\sigma(y^k,z;x^k)
+\frac{\sigma}{2}\|z-z^k\|_{T_z}^2\right\},\\
\bar{x}^{k+1}
&=x^k+\sigma(Ay^k+B\bar{z}^{k+1}-c),\\
\bar{y}^{k+1}
&=\mathop{\rm argmin}_{{y}}\left\{L_\sigma(y,\bar{z}^{k+1};\bar{x}^{k+1})
+\frac{\sigma}{2}\|y-y^k\|_{T_y}^2\right\},\\
w^{k+1}
&=(1-\rho)w^k+\rho\bar{w}^{k+1},\qquad \rho\in(0,2],\qquad k\geq 0,
\end{aligned}
\end{equation}
where $\bar{w}^{k+1}:=(\bar{y}^{k+1},\bar{z}^{k+1},\bar{x}^{k+1})$.
Let
\begin{equation*}
\mathcal{M} :=
\begin{pmatrix}
\sigma A^*A+\sigma T_y & 0 & A^*\\
0 & \sigma T_z & 0\\
A & 0 & \sigma^{-1}I
\end{pmatrix},
\qquad \textrm{and}\qquad
{\mathcal{J}_{\mathcal{T}}^{\mathcal{M}}}
=(\mathcal{M}+\mathcal{T})^{-1}\mathcal{M}.
\end{equation*}
{It has been shown in \cite{syzz2024acc-pADMM} that}  $0\in\mathcal{T}(\bar{w}^{k+1})+\mathcal{M}(\bar{w}^{k+1}-w^k)$ {is equivalent to}
\begin{equation*}
\left\{
\begin{aligned}
0&\in\partial g(\bar{z}^{k+1})+B^*(x^k+\sigma(Ay^k+B\bar{z}^{k+1}-c))
+\sigma T_z(\bar{z}^{k+1}-z^k),\\
\bar{x}^{k+1}
&=x^k+\sigma(Ay^k+B\bar{z}^{k+1}-c),\\
0&\in\partial f(\bar{y}^{k+1})+A^*\bar{x}^{k+1}
+\sigma(A^*A+T_y)(\bar{y}^{k+1}-y^k)+A^*(\bar{x}^{k+1}-x^k),
\end{aligned}
\right.
\end{equation*}
{which gives the first three steps of \eqref{eq-padmm-template}.} Thus $\bar{w}^{k+1}
=\mathcal{J}_{\mathcal{T}}^{\mathcal{M}}(w^k)${,
i.e., \eqref{eq-padmm-template} is a preconditioned generalized PPA. The operator $\mathcal{M}$ is positive semidefinite; since $M_y, M_z\succ0$, the two subproblems have unique solutions for every $w^k$, so $\mathcal{M}$ is an admissible preconditioner for $\mathcal{T}$.}

The Halpern{-}accelerated {iteration} is
\begin{equation*}
\begin{aligned}
w^{k+1}
&=\beta_k w^0+(1-\beta_k)\big[(1-\rho)w^k+\rho\bar{w}^{k+1}\big],
\end{aligned}
\end{equation*}
where $\beta_k=\frac{1}{k+r}$, for a given $r\geq2$.
We consider {the following} inexact {iteration}:
\begin{subequations}\label{eq-padmm-inexact}
\begin{align}
\text{Find}~ &\widetilde{w}^{k+1}~ \text{such that}\
\|\widetilde{w}^{k+1}-\mathcal{J}_{\mathcal{T}}^{\mathcal{M}}
w^k\|
=\|\widetilde{w}^{k+1}-\bar{w}^{k+1}\|\leq\varepsilon_k,
\label{eq-padmm-criterion}\\
w^{k+1}
&=\beta_k w^0+(1-\beta_k)
\big[(1-\rho)w^k+\rho\widetilde{w}^{k+1}\big],\qquad k\geq0.
\end{align}
\end{subequations}
This is Algorithm \ref{algo-iAccGPPA} with $\mathcal{A}=\mathcal{T}$, using $w^k$ and $\widetilde{w}^{k+1}$ in place of $\boldsymbol{x}^k$ and $\boldsymbol{y}^k$, respectively. Here $\widetilde{w}^{k+1}:=(\widetilde{y}^{k+1},\widetilde{z}^{k+1},\widetilde{x}^{k+1})$. We next show how to check the inexactness condition \eqref{eq-padmm-criterion} by using the subproblem residuals.

Let
\begin{equation*}
l_z^k:=-B^*(Ay^k+Bz^k-c+\sigma^{-1}x^k)+M_z z^k.
\end{equation*}
We assume that the inexactly computed $\widetilde{z}^{k+1}$ satisfies the condition
\begin{equation}\label{eq-padmm-z-residual}
\sigma d_z^k\in\partial g(\widetilde{z}^{k+1})
+\sigma M_z(\widetilde{z}^{k+1}-M_z^{-1}l_z^k),
\end{equation}
where $\sigma d_z^k$ denotes the residual vector incurred in solving the first subproblem in \eqref{eq-padmm-template}. Set
\begin{equation*}
\begin{aligned}
\widetilde{x}^{k+1}&:=x^k+\sigma(Ay^k+B\widetilde{z}^{k+1}-c),\\
\widetilde{l}_y^k&:=-A^*(Ay^k+B\widetilde{z}^{k+1}-c+\sigma^{-1}\widetilde{x}^{k+1})+M_y y^k.
\end{aligned}
\end{equation*}
We assume that the inexactly computed $\widetilde{y}^{k+1}$ satisfies the condition
\begin{equation}\label{eq-padmm-y-residual}
\sigma d_y^k \in\partial f(\widetilde{y}^{k+1})
+\sigma M_y\big(\widetilde{y}^{k+1}-M_y^{-1}\widetilde{l}_y^k\big),
\end{equation}
where $\sigma d_y^k$ denotes the residual vector incurred in solving the second subproblem in \eqref{eq-padmm-template}
with $\bar{z}^{k+1}$ and $\bar{x}^{k+1}$ replaced by
$\widetilde{z}^{k+1}$ and $\widetilde{x}^{k+1}$, respectively.

\begin{proposition}\label{prop-padmm-inexactness}\textup{(\textbf{Inexactness criterion})}
Suppose that $\widetilde{z}^{k+1}$ and $\widetilde{y}^{k+1}$ satisfy \eqref{eq-padmm-z-residual} and \eqref{eq-padmm-y-residual}, respectively. Then, it holds that
\begin{equation}\label{eq-padmm-errorbound}
\begin{aligned}
\|\widetilde{w}^{k+1}-\bar{w}^{k+1}\|^2
&\leq\left(\sigma^2+\|M_z^{-1/2}\|_2^2\right)\|M_z^{-1/2}d_z^k\|^2\\
&\qquad+\|M_y^{-1/2}\|_2^2
\left(\|M_y^{-1/2}d_y^k\|+2\|M_z^{-1/2}d_z^k\|\right)^2.
\end{aligned}
\end{equation}
In particular, if $d_z^k=0$, then
\begin{equation}\label{eq-padmm-exact-z}
\|\widetilde{w}^{k+1}-\bar{w}^{k+1}\|
\leq\|M_y^{-1/2}\|_2\|M_y^{-1/2}d_y^k\|.
\end{equation}
\end{proposition}
\begin{proof}
For $H\succ0$ and a closed proper convex function $h$, define
\begin{equation*}
\operatorname{Prox}_{h}^{H}(u):=\operatorname*{argmin}_{z}
\left\{h(z)+\frac{1}{2}\|z-u\|_{H}^{2}\right\}.
\end{equation*}
This mapping is firmly nonexpansive in the $H$-norm \cite[Proposition 2.3]{LS1997}:
\begin{equation}\label{eq-padmm-prox}
\|\operatorname{Prox}_{h}^{H}(v)-\operatorname{Prox}_{h}^{H}(u)\|_{H}^{2}
\leq\langle v-u,\operatorname{Prox}_{h}^{H}(v)-\operatorname{Prox}_{h}^{H}(u)\rangle_{H}.
\end{equation}
We apply this inequality to two subproblems and track the error through the multiplier update.

\begin{enumerate}[label=(\alph*),leftmargin=*]
\item From the optimality condition of the $z$-subproblem and \eqref{eq-padmm-z-residual}, we have that
\begin{equation*}
\bar{z}^{k+1}=\operatorname{Prox}_{g/\sigma}^{M_{z}}(M_{z}^{-1}l_{z}^{k}),
\qquad
\widetilde{z}^{k+1}=\operatorname{Prox}_{g/\sigma}^{M_{z}}\big(M_{z}^{-1}(l_{z}^{k}+d_{z}^{k})\big).
\end{equation*}
Applying \eqref{eq-padmm-prox} with $H=M_{z}$, $h=g/\sigma$, $u=M_{z}^{-1}l_{z}^{k}$, and $v=M_{z}^{-1}(l_{z}^{k}+d_{z}^{k})$, we obtain
\begin{equation*}
\begin{aligned}
\|\widetilde{z}^{k+1}-\bar{z}^{k+1}\|_{M_{z}}^{2}
&\leq\langle M_{z}^{-1}d_{z}^{k},\widetilde{z}^{k+1}-\bar{z}^{k+1}\rangle_{M_{z}}
=\langle d_{z}^{k},\widetilde{z}^{k+1}-\bar{z}^{k+1}\rangle\\
&=\langle M_{z}^{-1/2}d_{z}^{k},M_{z}^{1/2}(\widetilde{z}^{k+1}-\bar{z}^{k+1})\rangle
\leq\|M_{z}^{-1/2}d_{z}^{k}\|\,\|\widetilde{z}^{k+1}-\bar{z}^{k+1}\|_{M_{z}}.
\end{aligned}
\end{equation*}
Then, if $\|\widetilde{z}^{k+1}-\bar{z}^{k+1}\|_{M_{z}}>0$, dividing both sides by this norm yields
\begin{equation*}
\|\widetilde{z}^{k+1}-\bar{z}^{k+1}\|_{M_{z}}
\leq\|M_{z}^{-1/2}d_{z}^{k}\|.
\end{equation*}
The same bound holds trivially when $\widetilde{z}^{k+1}=\bar{z}^{k+1}$.

\item The multiplier updates give $\widetilde{x}^{k+1}-\bar{x}^{k+1}=\sigma B(\widetilde{z}^{k+1}-\bar{z}^{k+1})$. Thus,
\begin{equation*}
\begin{aligned}
\|\widetilde{x}^{k+1}-\bar{x}^{k+1}\|
&=\sigma\|BM_{z}^{-1/2}M_{z}^{1/2}(\widetilde{z}^{k+1}-\bar{z}^{k+1})\|
\leq\sigma\|BM_{z}^{-1/2}\|_{2}\,\|\widetilde{z}^{k+1}-\bar{z}^{k+1}\|_{M_{z}}\\
&\leq\sigma\|\widetilde{z}^{k+1}-\bar{z}^{k+1}\|_{M_{z}}
\leq\sigma\|M_{z}^{-1/2}d_{z}^{k}\|,
\end{aligned}
\end{equation*}
where $\|BM_{z}^{-1/2}\|_{2}\leq1$ follows from $M_{z}\succeq B^{*}B$, and the last inequality uses part (a).

\item Define
\begin{equation*}
l_{y}^{k}:=-A^{*}(Ay^{k}+B\bar{z}^{k+1}-c+\sigma^{-1}\bar{x}^{k+1})+M_{y}y^{k}.
\end{equation*}
The multiplier identity in part (b) gives
\begin{equation*}
\widetilde{l}_{y}^{k}-l_{y}^{k}
=-A^{*}B(\widetilde{z}^{k+1}-\bar{z}^{k+1})-\sigma^{-1}A^{*}(\widetilde{x}^{k+1}-\bar{x}^{k+1})
=-2A^{*}B(\widetilde{z}^{k+1}-\bar{z}^{k+1}).
\end{equation*}
Moreover, from the optimality condition of the $y$-subproblem and \eqref{eq-padmm-y-residual}, we have
\begin{equation*}
\bar{y}^{k+1}=\operatorname{Prox}_{f/\sigma}^{M_{y}}(M_{y}^{-1}l_{y}^{k}),
\qquad
\widetilde{y}^{k+1}=\operatorname{Prox}_{f/\sigma}^{M_{y}}\big(M_{y}^{-1}(\widetilde{l}_{y}^{k}+d_{y}^{k})\big).
\end{equation*}
Applying \eqref{eq-padmm-prox} as in part (a), we obtain
\begin{equation*}
\begin{aligned}
\|\widetilde{y}^{k+1}-\bar{y}^{k+1}\|_{M_{y}}
&\leq\|M_{y}^{-1/2}(\widetilde{l}_{y}^{k}-l_{y}^{k}+d_{y}^{k})\|\\
&\leq\|M_{y}^{-1/2}d_{y}^{k}\|+2\|M_{y}^{-1/2}A^{*}BM_{z}^{-1/2}\|_{2}\,\|\widetilde{z}^{k+1}-\bar{z}^{k+1}\|_{M_{z}}\\
&\leq\|M_{y}^{-1/2}d_{y}^{k}\|+2\|M_{z}^{-1/2}d_{z}^{k}\|,
\end{aligned}
\end{equation*}
where we used part (a) and $\|M_{y}^{-1/2}A^{*}BM_{z}^{-1/2}\|_{2}\leq\|M_{y}^{-1/2}A^{*}\|_{2}\|BM_{z}^{-1/2}\|_{2}\leq1$, which follows from $M_{y}\succeq A^{*}A$ and $M_{z}\succeq B^{*}B$.
\end{enumerate}
Combining (a)--(c) with $\|v\|\leq\|H^{-1/2}\|_{2}\|v\|_{H}$ gives \eqref{eq-padmm-errorbound}. Setting $d_{z}^{k}=0$ gives \eqref{eq-padmm-exact-z}. \hfill
\end{proof}

The following corollary establishes convergence and a nonergodic $\mathcal{O}(1/k)$ KKT residual rate for the inexact iterates $\{\widetilde{w}^{k+1}\}$.

\begin{corollary}[Convergence and KKT residual rate]\label{cor-padmm-kkt}
Let $\{w^k\}$ and $\{\widetilde{w}^{k+1}\}$ be generated by \eqref{eq-padmm-inexact} with $\rho\in(0,2]$ and $\beta_k=1/(k+r)$, $r\geq2$. Choose $d_y^k,d_z^k$ so that the right-hand side of \eqref{eq-padmm-errorbound} is at most $\varepsilon_k^2$. If
\begin{equation*}
\sum_{k=0}^{\infty}(k+r-1)^2\varepsilon_k^2<\infty
\quad\text{for }0<\rho<2,
\qquad\text{or}\qquad
\sum_{k=0}^{\infty}(k+r-1)\varepsilon_k<\infty
\quad\text{for }\rho=2,
\end{equation*}
then $\widetilde{w}^{k+1}$ converges to a KKT point $w^*\in\operatorname{zero}(\mathcal{T})$, and
\begin{equation*}
\operatorname{dist}\big(0,\mathcal{T}(\widetilde{w}^{k+1})\big)=\mathcal{O}(1/k).
\end{equation*}
Moreover, $w^k\to w^*$ if $0<\rho<2$ or $\mathcal{M}\succ0$.
\end{corollary}
\begin{proof}
By Proposition \ref{prop-padmm-inexactness}, the inexactness condition \eqref{eq-padmm-criterion} holds. Each stated weighted summability condition implies $\sum_{k=0}^{\infty}\varepsilon_k<\infty$ and $\varepsilon_k=o(1/k)$. Theorem \ref{thm-preconditioned-convergence} gives $\widetilde{w}^{k+1}\to w^*$ and, for $0<\rho<2$, $w^k\to w^*$. When $\mathcal{M}\succ0$, Corollary \ref{rem-positive-definite} also gives $w^k\to w^*$ for $\rho=2$.

By the subproblem residual inclusions \eqref{eq-padmm-z-residual}--\eqref{eq-padmm-y-residual} and the multiplier update, we have
\begin{equation*}
v^{k+1}:=\mathcal{M}(w^{k}-\widetilde{w}^{k+1})+\sigma(d_{y}^{k},d_{z}^{k},0)
\in\mathcal{T}(\widetilde{w}^{k+1}).
\end{equation*}
Since $M_y,M_z\succ0$ and $\sigma>0$, the assumption that the right-hand side of \eqref{eq-padmm-errorbound} does not exceed $\varepsilon_k^2$ also implies $\|(d_y^k,d_z^k)\|\leq K_d\varepsilon_k$ for a constant $K_d$ independent of $k$. Splitting $w^{k}-\widetilde{w}^{k+1}=(w^{k}-\bar{w}^{k+1})+(\bar{w}^{k+1}-\widetilde{w}^{k+1})$ and using $\|\mathcal{M}h\|\leq\sqrt{\|\mathcal{M}\|_{2}}\|h\|_{\mathcal{M}}$, we obtain
\begin{equation*}
\begin{aligned}
\|v^{k+1}\|
&\leq\|\mathcal{M}(w^{k}-\bar{w}^{k+1})\|+\|\mathcal{M}\|_{2}\|\bar{w}^{k+1}-\widetilde{w}^{k+1}\|+\sigma\|(d_{y}^{k},d_{z}^{k})\|\\
&\leq\sqrt{\|\mathcal{M}\|_{2}}\,\|w^{k}-\mathcal{J}_{\mathcal{T}}^{\mathcal{M}}w^{k}\|_{\mathcal{M}}+(\|\mathcal{M}\|_{2}+\sigma K_d)\varepsilon_{k}.
\end{aligned}
\end{equation*}
Theorem \ref{thm-preconditioned-iAccGPPA}(i)--(ii) and $\varepsilon_k=o(1/k)$ give $\|v^{k+1}\|=\mathcal{O}(1/k)$. The KKT residual rate follows from $v^{k+1}\in\mathcal{T}(\widetilde{w}^{k+1})$.
\hfill
\end{proof}

\subsection{Inexact accelerated preconditioned PDHG}\label{sec-App-PDHG}

Consider the following convex-concave saddle-point problem:
\begin{equation}\label{eq-pdhg-P}
\min_y\max_x\; f(y)+\langle Ky,x\rangle-g^*(x).
\end{equation}
Here $f$ and $g$ are closed proper convex functions defined on finite-dimensional Euclidean spaces, $g^*$ is the Fenchel conjugate of $g$, and $K$ is a linear operator with adjoint $K^*$. Let $w=(y,x)$, where $y$ and $x$ are the primal and dual variables, respectively. The KKT conditions are
\begin{equation}\label{eq-pdhg-KKT}
0\in\mathcal{T}(w):=
\begin{pmatrix}
\partial f(y)+K^*x\\
\partial g^*(x)-Ky
\end{pmatrix}.
\end{equation}
We assume that \eqref{eq-pdhg-KKT} has a solution. The operator $\mathcal{T}$ is maximally monotone.
Let $T_y$ and $T_x$ be fixed self-adjoint positive definite operators such that
\begin{equation*}
\mathcal{M}:=
\begin{pmatrix}
T_y&-K^*\\
-K&T_x
\end{pmatrix}\succeq0.
\end{equation*}
We consider the following relaxed version of the preconditioned PDHG method in \cite{liu2021preconditioning} for solving \eqref{eq-pdhg-P}:
\begin{equation}\label{eq-pdhg-template}
\begin{aligned}
\bar{y}^{k+1}&=\mathop{\rm argmin}_y
\left\{f(y)+\langle Ky,x^k\rangle+\frac12\|y-y^k\|_{T_y}^2\right\},\\
\bar{x}^{k+1}&=\mathop{\rm argmin}_x
\left\{g^*(x)-\langle K(2\bar{y}^{k+1}-y^k),x\rangle+\frac12\|x-x^k\|_{T_x}^2\right\},\\
w^{k+1}&=(1-\rho)w^k+\rho\bar{w}^{k+1},\qquad\rho\in(0,2],
\end{aligned}
\end{equation}
where $\bar{w}^{k+1}=(\bar{y}^{k+1},\bar{x}^{k+1})$.
The optimality conditions for the first two steps in \eqref{eq-pdhg-template} give
$0\in\mathcal{T}(\bar{w}^{k+1})+\mathcal{M}(\bar{w}^{k+1}-w^k)$. Thus $\bar{w}^{k+1}=\mathcal{J}_{\mathcal{T}}^{\mathcal{M}}w^k$, i.e., \eqref{eq-pdhg-template} is a preconditioned generalized PPA. Since $T_y,T_x\succ0$, the two subproblems have unique solutions for every $w^k$, so $\mathcal{M}$ is an admissible preconditioner for $\mathcal{T}$.
The condition $\mathcal{M}\succeq0$ is equivalent to $\|T_x^{-1/2}KT_y^{-1/2}\|_2\leq1$. For $T_y=\sigma_y\mathcal{I}$ and $T_x=\sigma_x\mathcal{I}$ with $\sigma_y,\sigma_x>0$, \eqref{eq-pdhg-template} with $\rho=1$ reduces to the PDHG method~\cite{chambolle2011first}.

Fix $\rho\in(0,2]$ and $r\geq2$, set $\beta_{k}=1/(k+r)$, and choose an inexactness tolerance sequence $\{\varepsilon_k\}_{k=0}^{\infty}\subseteq[0,\infty)$. Starting from $w^0=(y^0,x^0)$, we consider the following inexact Halpern-accelerated iteration:
\begin{subequations}\label{eq-pdhg-inexact}
\begin{align}
\text{Find}\ &\widetilde{w}^{k+1}\ \text{such that}\quad
\|\widetilde{w}^{k+1}-\mathcal{J}_{\mathcal{T}}^{\mathcal{M}}w^k\|
=\|\widetilde{w}^{k+1}-\bar{w}^{k+1}\|\leq\varepsilon_k,
\label{eq-pdhg-criterion}\\
w^{k+1}&=\beta_k w^0+(1-\beta_k)
\big[(1-\rho)w^k+\rho\widetilde{w}^{k+1}\big],\qquad k\geq0,
\label{eq-pdhg-Halpern}
\end{align}
\end{subequations}
where $\widetilde{w}^{k+1}=(\widetilde{y}^{k+1},\widetilde{x}^{k+1})$. This is Algorithm \ref{algo-iAccGPPA} with $\mathcal{A}=\mathcal{T}$.
We next show how to check the inexactness condition \eqref{eq-pdhg-criterion} by using the subproblem residuals. Suppose that the computed points $\widetilde{y}^{k+1},\widetilde{x}^{k+1}$ and residual vectors $d_y^k,d_x^k$ satisfy
\begin{equation}\label{eq-pdhg-subproblem-residuals}
\begin{aligned}
d_y^k&\in\partial f(\widetilde{y}^{k+1})
+T_y(\widetilde{y}^{k+1}-y^k)+K^*x^k,\\
d_x^k&\in\partial g^*(\widetilde{x}^{k+1})
+T_x(\widetilde{x}^{k+1}-x^k)-K(2\widetilde{y}^{k+1}-y^k).
\end{aligned}
\end{equation}

\begin{proposition}\label{prop-pdhg-inexactness}\textup{(\textbf{Inexactness criterion})}
Suppose that $\widetilde{y}^{k+1}$ and $\widetilde{x}^{k+1}$ satisfy \eqref{eq-pdhg-subproblem-residuals}. Then
\begin{equation}\label{eq-pdhg-errorbound}
\|\widetilde{w}^{k+1}-\bar{w}^{k+1}\|^2
\leq\|T_y^{-1/2}\|_2^2\|T_y^{-1/2}d_y^k\|^2
+\|T_x^{-1/2}\|_2^2
\big(\|T_x^{-1/2}d_x^k\|+2\|T_y^{-1/2}d_y^k\|\big)^2.
\end{equation}
In particular, if $d_y^k=0$, then
\begin{equation}\label{eq-pdhg-exact-x}
\|\widetilde{w}^{k+1}-\bar{w}^{k+1}\|^2
\leq\|T_x^{-1/2}\|_2^2\,
\|T_x^{-1/2}d_x^k\|^2.
\end{equation}
\end{proposition}
\begin{proof}
The optimality conditions in \eqref{eq-pdhg-template} give
\begin{equation*}
\bar{y}^{k+1}=\operatorname{Prox}_{f}^{T_{y}}\big(y^{k}-T_{y}^{-1}K^{*}x^{k}\big),
\qquad
\bar{x}^{k+1}=\operatorname{Prox}_{g^{*}}^{T_{x}}\big(x^{k}+T_{x}^{-1}K(2\bar{y}^{k+1}-y^{k})\big).
\end{equation*}
The residual inclusions \eqref{eq-pdhg-subproblem-residuals} show that $\widetilde{y}^{k+1}$ and $\widetilde{x}^{k+1}$ are obtained by adding $T_{y}^{-1}d_{y}^{k}$ and $T_{x}^{-1}[d_{x}^{k}+2K(\widetilde{y}^{k+1}-\bar{y}^{k+1})]$ to the respective proximal arguments above. Applying \eqref{eq-padmm-prox} as in part (a) of the proof of Proposition \ref{prop-padmm-inexactness}, we obtain
\begin{equation*}
\begin{aligned}
\|\widetilde{y}^{k+1}-\bar{y}^{k+1}\|_{T_{y}}
&\leq\|T_{y}^{-1/2}d_{y}^{k}\|,\\
\|\widetilde{x}^{k+1}-\bar{x}^{k+1}\|_{T_{x}}
&\leq\|T_{x}^{-1/2}[d_{x}^{k}+2K(\widetilde{y}^{k+1}-\bar{y}^{k+1})]\|\\
&\leq\|T_{x}^{-1/2}d_{x}^{k}\|+2\|T_{x}^{-1/2}KT_{y}^{-1/2}\|_{2}\,\|\widetilde{y}^{k+1}-\bar{y}^{k+1}\|_{T_{y}}\\
&\leq\|T_{x}^{-1/2}d_{x}^{k}\|+2\|T_{y}^{-1/2}d_{y}^{k}\|,
\end{aligned}
\end{equation*}
where the last inequality uses $\|T_{x}^{-1/2}KT_{y}^{-1/2}\|_{2}\leq1$ and the first bound above. Combining these estimates with $\|v\|\leq\|H^{-1/2}\|_{2}\|v\|_{H}$ for $H=T_{y},T_{x}$ gives \eqref{eq-pdhg-errorbound}. Setting $d_{y}^{k}=0$ gives \eqref{eq-pdhg-exact-x}.
\end{proof}

\begin{corollary}[Convergence and KKT residual rate]\label{cor-pdhg-convergence}
Let $\{w^k\}$ and $\{\widetilde{w}^{k+1}\}$ be generated by \eqref{eq-pdhg-inexact} with $\rho\in(0,2]$ and $\beta_k=1/(k+r)$, $r\geq2$. Choose $d_y^k,d_x^k$ so that the right-hand side of \eqref{eq-pdhg-errorbound} is at most $\varepsilon_k^2$. If
\begin{equation*}
\sum_{k=0}^{\infty}(k+r-1)^2\varepsilon_k^2<\infty
\quad\text{for }0<\rho<2,
\qquad\text{or}\qquad
\sum_{k=0}^{\infty}(k+r-1)\varepsilon_k<\infty
\quad\text{for }\rho=2,
\end{equation*}
then $\widetilde{w}^{k+1}$ converges to a saddle point $w^*$ of \eqref{eq-pdhg-P}, and
\begin{equation*}
\operatorname{dist}\big(0,\mathcal{T}(\widetilde{w}^{k+1})\big)
=\mathcal{O}(1/k).
\end{equation*}
Moreover, $w^k\to w^*$ if $0<\rho<2$ or $\mathcal{M}\succ0$.
\end{corollary}
\begin{proof}
By Proposition \ref{prop-pdhg-inexactness}, the inexactness condition \eqref{eq-pdhg-criterion} holds. Each stated weighted summability condition implies $\sum_{k=0}^{\infty}\varepsilon_k<\infty$ and $\varepsilon_k=o(1/k)$. Theorem \ref{thm-preconditioned-convergence} gives $\widetilde{w}^{k+1}\to w^*$ and, for $0<\rho<2$, $w^k\to w^*$. When $\mathcal{M}\succ0$, Corollary \ref{rem-positive-definite} also gives $w^k\to w^*$ for $\rho=2$.

By the subproblem residual inclusions \eqref{eq-pdhg-subproblem-residuals}, we have
\begin{equation*}
v^{k+1}:=\mathcal{M}(w^{k}-\widetilde{w}^{k+1})+(d_{y}^{k},d_{x}^{k})
\in\mathcal{T}(\widetilde{w}^{k+1}).
\end{equation*}
Since $T_{y}$ and $T_{x}$ are fixed positive definite operators, the assumption that the right-hand side of \eqref{eq-pdhg-errorbound} is at most $\varepsilon_{k}^{2}$ implies $\|(d_{y}^{k},d_{x}^{k})\|=\mathcal{O}(\varepsilon_{k})$. Using the definition of $v^{k+1}$ and $\|\mathcal{M}h\|\leq\sqrt{\|\mathcal{M}\|_{2}}\|h\|_{\mathcal{M}}$, we obtain
\begin{equation*}
\begin{aligned}
\|v^{k+1}\|
&\leq\|\mathcal{M}(w^{k}-\bar{w}^{k+1})\|+\|\mathcal{M}\|_{2}\|\bar{w}^{k+1}-\widetilde{w}^{k+1}\|+\|(d_{y}^{k},d_{x}^{k})\|\\
&\leq\sqrt{\|\mathcal{M}\|_{2}}\,\|w^{k}-\mathcal{J}_{\mathcal{T}}^{\mathcal{M}}w^{k}\|_{\mathcal{M}}+\|\mathcal{M}\|_{2}\varepsilon_{k}+\|(d_{y}^{k},d_{x}^{k})\|.
\end{aligned}
\end{equation*}
Theorem \ref{thm-preconditioned-iAccGPPA}(i)--(ii) and $\varepsilon_k=o(1/k)$ give $\|v^{k+1}\|=\mathcal{O}(1/k)$. The KKT residual rate follows from $v^{k+1}\in\mathcal{T}(\widetilde{w}^{k+1})$.
\end{proof}

\section{Conclusions}\label{sec-conclusions}

This paper studies the \iHPGPPA for maximal monotone inclusion problems. We establish convergence of the inexact resolvent sequence and an $\mathcal{O}(1/k^2)$ convergence rate for the squared fixed-point residual in the $\mathcal{M}$-seminorm. Our analysis gives different sufficient inexactness conditions for $0<\rho<2$ and $\rho=2$ to preserve this accelerated residual rate. At $\rho=2$, the weighted summability condition is also necessary for guaranteeing this rate uniformly over all admissible error directions. For $0<\rho<2$, we also establish convergence of $\{\boldsymbol{x}^k\}$, even when $\mathcal{M}$ is singular. If $(\mathcal{M}+\mathcal{A})^{-1}$ is Lipschitz, the squared Euclidean fixed-point residual admits the same rate at $\boldsymbol{y}^{k}$ for all $\rho\in(0,2]$, and at $\boldsymbol{x}^{k}$ for $0<\rho<2$ under the corresponding weighted summability conditions. For accelerated pADMM and preconditioned PDHG, we derive inexactness criteria based on subproblem residuals and obtain a nonergodic $\mathcal{O}(1/k)$ KKT residual rate for the inexact iterates. We leave the study of adaptive restart strategies for future research.

\medskip
\noindent\textbf{Acknowledgements}\quad The research of Lei Yang is supported in part by the National Natural Science Foundation of China under grant 12301411. The research of Kim-Chuan Toh is supported by the Ministry of Education, Singapore, under its Academic Research Fund Tier 2 grant (MOE-T2EP20224-0017).

\bibliographystyle{plain}
\bibliography{Ref_iHP_GPPA}

\end{document}